\documentclass[11pt]{article}
\usepackage[a4paper,margin=1in]{geometry}
\usepackage{amsmath,amsthm,amssymb,mathtools,mathrsfs,bm}
\usepackage{enumitem}
\usepackage{hyperref}
\hypersetup{colorlinks=true,linkcolor=blue,citecolor=blue,urlcolor=blue}
\numberwithin{equation}{section}
\usepackage{hyperref}
\newtheorem{theorem}{Theorem}[section]
\newtheorem{lemma}[theorem]{Lemma}
\newtheorem{proposition}[theorem]{Proposition}
\newtheorem{corollary}[theorem]{Corollary}
\theoremstyle{remark}
\newtheorem{remark}[theorem]{Remark}

\title{Local Turnpike Stability for Two-Population Second-Order Mean-Field Game Systems }

 \author{Mohamed Ghattassi\footnote{Department of Mathematics, New York University in Abu Dhabi, Saadiyat Island, P.O. Box 129188, Abu Dhabi, United Arab Emirates. {\sf Email:  mg6888@nyu.edu}} \and Fanze Kong \thanks{Department of Applied Mathematics, University of Washington, Seattle, WA 98195, USA; {\sf Email: fzkong@uw.edu}} \and  Nader Masmoudi \footnote{Department of Mathematics, New York University in Abu Dhabi, Saadiyat Island, P.O. Box 129188, Abu Dhabi, United Arab Emirates-- Courant Institute of Mathematical Sciences, New York University, 251 Mercer Street, New York, NY 10012, USA. {\sf Email:  nm30@nyu.edu}}}

\begin{document}

\maketitle

\begin{abstract}

This paper establishes a local turnpike property for a two-population, second-order mean-field game system with a quadratic Hamiltonian, in which the running cost for each population may depend on both population densities. In the case of two populations, imposing a scalar stability condition separately on each species is not sufficient because it neglects cross-interaction effects. The appropriate requirement is instead a matrix-valued coercivity condition on vector perturbations $(\mu_1,\mu_2)$, which involves the symmetric part of the linearized interaction matrix together with weighted Poincaré inequalities associated with the two stationary measures. Consequently, in the multi-population framework, diffusion still acts independently on each population, but the interactions are controlled by the full coupling matrix. A naive scalar energy functional fails to capture mixed second-variation terms arising from costs that depend on both densities. We prove that matrix coercivity is precisely equivalent to the strict positivity of this second variation, and we provide a spectral characterization in terms of the principal spectral value of the coupled stationary linearized operator. For linear models with cross-interactions, this yields explicit algebraic conditions that separate stable parameter ranges from unstable ones. In the non-potential framework, the skew-symmetric part of the interaction matrix leaves the coercivity threshold unchanged, but it can induce transient oscillatory dynamics. Moreover, the scalar turnpike theory extends in a natural, though nontrivial, manner to multi-population systems, where stability is governed by the positivity of a coupled matrix-valued second variation that encodes both self-interactions and cross-interactions.
\end{abstract}

\tableofcontents

\section{Introduction}
In this paper, we investigate a finite-horizon mean-field game model for two interacting populations as follows
\begin{align}\label{timedepre}
\left\{
\begin{array}{ll}
-\Delta u_1+\frac12|\nabla u_1|^2
=
\partial_tu_1+f_1(m_1,m_2),&x\in \mathbb T^N, t\in(0,T),
\\
\partial_t m_1
=
\Delta m_1+\nabla\cdot(m_1\nabla u_1),&x\in \mathbb T^N, t\in(0,T)
\\
-\Delta u_2+\frac12|\nabla u_2|^2
=
\partial_tu_2+f_2(m_1,m_2),&x\in \mathbb T^N, t\in(0,T),
\\
\partial_t m_2
=
\Delta m_2+\nabla\cdot(m_2\nabla u_2),&x\in \mathbb T^N, t\in(0,T),
\\
u_i(T,x)=u_{iT}(x),\qquad m_i(0,x)=m_{i0}(x),&x\in \mathbb T^N,
~~ i=1,2,
\end{array}
\right.
\end{align}
where $T>0$ denotes the time horizon,
$f_i:\mathbb R\times\mathbb{R}\to\mathbb{R}$, $i=1,2$ is a $C^2$ coupling function
whose first-order and second-order partial derivatives are bounded.  All data are assumed
to be periodic in the spatial variable $x$.
The initial distributions $m_{i0}$ and the terminal costs $u_{iT},$ $i=1,2$ are prescribed
functions belonging to $C^{2,\alpha}(\mathbb T^N)$.

We briefly introduce the derivation of multi-population  mean-field game systems.  As shown in \cite{Lasry,feleqi2013derivation}, assume that the $j$-th agent in the $i$-th population, with $i=1,2$ and $j=1,\ldots,n_i$, satisfies the following controlled stochastic differential equation (SDE):
\begin{align*}
dX_t^{i,j}
=
-\alpha_t^{i,j}\,dt
+
\sqrt{2}\,dB_t^{i,j},
\qquad
X_0^{i,j}=x^{i,j}\in\mathbb T^N,
\end{align*}
where $x^{i,j}$ is the initial state, $\alpha^{i,j}_t$ denotes the controlled velocity and $B_t^{i,j}$ represent the independent Brownian motion. Assuming that all agents within the $i$-th population are indistinguishable, we may drop the index $j$ and denote $X_t^{i,j}$ and $\alpha_t^{i,j}$ by $X_t^i$ and $\alpha_t^i$, respectively. Each agent in the $i$-th population then seeks to minimize the following expected cost: 
\begin{equation}\label{longsenseexpectation}
\begin{aligned}
J^i(\alpha^i;m_1,m_2)
=
\mathbb E\Bigg[
&\int_0^T
\left(
\frac12|\alpha_t^i|^2
+
 f_i\bigl(m_1(t,X_t^i),m_2(t,X_t^i)\bigr)
\right)\,dt
\\
&\qquad
+u_{iT}(X_T^i)
\Bigg],
\qquad i=1,2,
\end{aligned}
\end{equation}
where the function $f_i$ depends on the population density $m_1$ and $m_2$.  By applying the standard dynamic programming principle, the coupled PDE system consisting of Hamilton-Jacobi-Bellman equations and Fokker-Planck equations is formulated in (\ref{timedepre}). By investigating the long-time behavior of solutions to (\ref{timedepre}) and considering the ansatz $\displaystyle u_i(t,x)=\lambda_i(T-t)+\bar u_i(x)$, $\displaystyle m_i(t,x)=\bar m_i(x),$
, the associated ergodic mean-field game system is formally formulated as follows:
\begin{align}\label{ss1-recalled}
\left\{
\begin{array}{ll}
-\Delta \bar u_1+\frac12|\nabla \bar u_1|^2+\lambda_1
=
 f_1(\bar m_1,\bar m_2),&x\in \mathbb T^N,
\\
\Delta \bar m_1+\nabla\cdot(\bar m_1\nabla \bar u_1)=0,&x\in \mathbb T^N,
\\
-\Delta \bar u_2+\frac12|\nabla \bar u_2|^2+\lambda_2 
=
  f_2(\bar m_1,\bar m_2),&x\in \mathbb T^N,
\\
\Delta \bar m_2+\nabla\cdot(\bar m_2\nabla \bar u_2)=0,&x\in \mathbb T^N,
\\
\int_{\mathbb T^N} \bar m_1\,dx=\int_{\mathbb T^N}m_{10}\,dx=1,\int_{\mathbb T^N} \bar m_2\,dx=\int_{\mathbb T^N}m_{20}\,dx=1.
\end{array}
\right.
\end{align}
A natural question is whether the ergodic equilibrium
$(\bar u_1,\bar u_2,\bar m_1,\bar m_2)$
locally attracts nearby time-dependent solution
$(u_1,u_2,m_1,m_2)$ to (\ref{ss1-recalled})
on the interval $(0,T)$, provided that the initial-terminal data are sufficiently close to the stationary state.
This behavior is commonly known as the \emph{local turnpike property}.
Our objective is to establish such local stability results without imposing any global uniqueness assumption on stationary equilibria.  We remark that local turnpike stability is commonly established under the Lasry-Lions monotonicity condition.  Our goal is to understand the local turnpike stability of mean-field game systems in the absence of the standard monotonicity condition.

The present study is directly inspired by the local turnpike theory developed by Cirant and De Bernardi for mean-field games (MFG) and mean-field control problems with quadratic Hamiltonians \cite{cirant2025local}. Nonetheless, the contribution of this work does not amount to a straightforward extension of their scalar framework. The central difference resides in the structural properties of the systems under investigation: whereas Cirant and De Bernardi focus on a \emph{single-population} model, the current analysis addresses a \emph{multi-population} MFG system, in which two distinct distributions interact through potentially nontrivial cross-coupling mechanisms. In the single-population setting, the finite-horizon MFG system is given by
$$
\left\{
\begin{aligned}
-\partial_t u - \Delta u + \frac12|\nabla u|^2
&=
f(x,m),
\\
\partial_t m - \Delta m - \nabla\cdot(m\nabla u)
&= 0,
\end{aligned}
\right.
$$
where $(u,m)$ denotes, respectively, the value function and the population density. A stationary solution is a pair $(\bar u,\bar m)$, and the local turnpike analysis is conducted in a neighborhood of this stationary equilibrium. A central hypothesis in the work of Cirant and De Bernardi is a scalar second-variation-type condition of the form
$$
\int_{ \mathbb{T}^N}
f_m(x,\bar m)\,\mu^2
+
\bar m
\left|
\nabla\left(\frac{\mu}{\bar m}\right)
\right|^2
\,dx
\geq
\eta
\int_{ \mathbb{T}^N}
\bar m
\left|
\nabla\left(\frac{\mu}{\bar m}\right)
\right|^2
\,dx,
$$
for all mass-preserving perturbations $\mu$, i.e., perturbations satisfying $\int_ \mathbb{T}^N \mu\,dx = 0$. This condition is intrinsically scalar: the linearized coupling is fully encoded by the single coefficient $f_m(x,\bar m)$. The centerpiece of their analysis is to demonstrate that this local second-order positivity condition, together with the structural relation $ \nabla \bar m = -\bar m \nabla \bar u$, entails that finite-horizon solutions starting sufficiently near the stationary state converge to it at an exponential rate over most of the time interval $[0,T]$.

In contrast, the system investigated in the present work is given by (\ref{timedepre}).
In this framework, the density variable is no longer scalar-valued but instead takes values in $\mathbb{R}^2$ and is described by the vector field $\mathbf m = (m_1, m_2)$. Consequently, the coupling term becomes vector-valued and can be written as $\mathbf f(\mathbf m) = \bigl(f_1(m_1,m_2),\, f_2(m_1,m_2)\bigr)$. The linearization of this coupling around a stationary state $\bar{\mathbf m} = (\bar m_1,\bar m_2)$ is then not characterized by a scalar coefficient but rather by the Jacobian matrix
\[
A(x)
=
D_{\mathbf m}\mathbf f\bigl(\bar{\mathbf m}(x)\bigr)
=
\begin{pmatrix}
\partial_{m_1}f_1(\bar{\mathbf m}) &
\partial_{m_2}f_1(\bar{\mathbf m})
\\[4pt]
\partial_{m_1}f_2(\bar{\mathbf m}) &
\partial_{m_2}f_2(\bar{\mathbf m})
\end{pmatrix}.
\]
This constitutes the first major structural difference with respect to the scalar case. In the single-population theory, stability is governed by $f_m(x,\bar m)$. In the two-population framework, the relevant stability object is instead the symmetric part of the weighted interaction matrix, $ \displaystyle DA_s(x)  := \frac{DA(x) + A(x)^TD}{2}$ with the weight matrix $D=\text{diag}\{d_1,d_2\}$ and $d_1,d_2>0$. The natural analogue of the scalar stability condition is therefore the following matrix coercivity requirement:
$$
\int_{ \mathbb{T}^N}
\left[
\boldsymbol\mu^T DA_s(x)\boldsymbol\mu
+
\sum_{i=1}^2d_i
\bar m_i
\left|
\nabla\left(\frac{\mu_i}{\bar m_i}\right)
\right|^2
\right]dx
\geq
\eta
\sum_{i=1}^2d_i
\int_{ \mathbb{T}^N}
\bar m_i
\left|
\nabla\left(\frac{\mu_i}{\bar m_i}\right)
\right|^2dx,
$$
for all perturbations
$$
\boldsymbol\mu = (\mu_1,\mu_2),
\quad
\int_{\mathbb T^N}\mu_i\,dx = 0,
\qquad i=1,2,\quad \mu_i \in L^1(\mathbb T^N)\cap L^\infty(\mathbb T^N)
$$
with $ \displaystyle \frac{\mu_i}{\bar m_i}\in W^{1,2}(\bar m_i\,dx)$ and $\displaystyle \frac{\mu_i}{\bar m_i}\in L^\infty(\mathbb T^N)$.
This coercivity hypothesis has no counterpart in the scalar theory. It provides the natural two-population extension of the scalar second-variation condition introduced by Cirant and De Bernardi \cite{cirant2025local}. We also note that a related positivity condition for a symmetrized Jacobian-type matrix arising in two-population MFG systems was used to derive a priori estimates; see, for example, \cite{cirant2015multipopulation}.

The main contribution of this article is thus the identification and thorough examination of the correct \emph{matrix-valued} stability mechanism that governs local turnpike behavior in multi-population MFG systems. This result cannot be derived by simply applying the scalar theorem to each component separately. A component-by-component argument  would neglect the cross-interaction terms $ \partial_{m_2}f_1(\bar m_1,\bar m_2)\mu_2$, $\partial_{m_1}f_2(\bar m_1,\bar m_2)\mu_1$, which couple the two Hamilton-Jacobi equations with the two Fokker-Planck equations and may either enhance or attenuate the stability properties of the system.  Therefore, the analysis must be conducted at the level of the entire two-population system as a whole.   Concerning the local turnpike stability of (\ref{timedepre}), we summarize our main result as the following theorem.
\begin{theorem}
\label{thm:twopopturnpike}
Assume that \eqref{S1}, \eqref{S2}, and \eqref{matrixS2} hold. Let $\mathbf f\in C^2(\mathcal U)$; then there exist constants $\gamma>0$, $C>0$, and $\sigma>0$ independent of $T$ such that whenever
$$
\sum_{i=1}^2
\|m_{i0}-\bar m_i\|_{L^\infty}
+
\sum_{i=1}^2
\|u_{iT}-\bar u_i\|_{W^{1,\infty}}
\leq \gamma,
$$
and
$$
\int_{\mathbb T^N}m_{i0}\,dx=1,
\qquad m_{i0}>0,
$$
there exists at least one classical solution of \eqref{timedepre} satisfying
$$
\sum_{i=1}^{2} \|m_i(t,\cdot)-\bar m_i\|_{L^\infty(\mathbb T^N)}
\leq
C_D\gamma
\left(
e^{-\sigma_D t}
+
e^{-\sigma_D(T-t)}
\right),
$$
and
$$
\sum_{i=1}^2
\left\|
\sqrt{\bar m_i}
\left(
\nabla u_i(t,\cdot)-\nabla\bar u_i
\right)
\right\|_{L^2(\mathbb T^N)}
\leq
C_D\gamma
\left(
e^{-\sigma_D t}
+
e^{-\sigma_D(T-t)}
\right),
$$
where constants $C_D>0$ and $\sigma_D>0$ depending on $d_1,d_2>0.$
\end{theorem}


The study of the turnpike phenomenon in optimal control has developed in parallel with research on the long-time behavior of solutions in mean-field Game (MFG) theory. In this context, it is natural to single out the seminal contributions \cite{MR2928380,MR3103242}, as well as more recent advances such as \cite{MR4294187,cirant2024long,ersland2025long}. A central unifying aspect of these works is that their assumptions are designed to produce a global turnpike property: both the time-dependent and stationary problems admit unique solutions, and the former converge to the latter as the time horizon tends to infinity, independently of the specific initial or terminal conditions. Typically, these assumptions are expressed as monotonicity conditions or, more generally, as smallness constraints on the problem data.  More recently, a stability analysis was performed in \cite{ghattassi2023stability}, where the authors investigate the long-time behavior of small perturbations of spatially homogeneous states in elongated, corridor-shaped domains. They show that spatially uniform flow configurations are nonlinearly stable in the monotone regime. Their approach is based on a Fourier analysis of the linearized system, combined with a fixed-point argument for the nonlinear generalized Hughes model \cite{MR4913256}.

Our work extends the two-population version of the scalar second-variation condition introduced by Cirant and De Bernardi \cite{cirant2025local}. In that paper, the authors analyze problems without assuming the standard monotonicity condition, a setting in which the existence of several stationary states is usually expected. The stability of stationary equilibria in such non-monotone regimes has so far been investigated only in a relatively small number of contributions \cite{MR4661001,MR4701427}, and the available results remain largely confined to rather specific classes of models. In addition, once monotonicity is dropped, dynamical behaviors that significantly differ from the turnpike phenomenon have been documented, including the emergence of periodic solutions \cite{MR3944249}. The central contribution of this article is therefore the identification and detailed analysis of the appropriate \emph{matrix-valued} stability mechanism that dictates local turnpike behavior in multi-population MFG systems. This mechanism cannot be obtained by merely applying the scalar theorem to each component in isolation. A componentwise argument would overlook the cross-interaction terms $ \partial_{m_2}f_1(\bar m_1,\bar m_2)\mu_2$, $\partial_{m_1}f_2(\bar m_1,\bar m_2)\mu_1$, which link the two Hamilton–Jacobi equations with the two Fokker–Planck equations and can either reinforce or weaken the stability properties of the full system. Consequently, the analysis must be carried out at the level of the coupled two-population system as a whole.

The paper is organized as follows. Section~\ref{secpre} introduces the two-population mean-field game system and recalls the main assumptions on the Hamiltonian, the coupling terms, and the stationary density constraints. In Section~\ref{sectwopopturnpike}, we establish the coercivity condition and prove the local turnpike stability result. Section~\ref{secspectralc} provides a spectral interpretation of assumption \eqref{matrixS2}. In particular, we show that this coercivity condition is tightly linked to the eigenvalue problem associated with the linearized elliptic system. Section~\ref{secpotentialenergy} investigates how \eqref{matrixS2} relates to the second variation of the variational formulation corresponding to potential two-population mean-field games. In Section~\ref{seccrossnonpotential}, we give several examples illustrating cross-interaction thresholds and non-potential interactions. We conclude the paper by presenting a summary of our findings and outlining potential directions for future research.

\section{Preliminaries}\label{secpre}
This section is devoted to introducing the notation and the assumptions on the coupling functions and ergodic solutions used throughout the paper. First of all, we define $\bar{\mathbf m} =(\bar m_1,\bar m_2)$, $\bar{\mathbf u} =(\bar u_1,\bar u_2)$, and $\boldsymbol\lambda =(\lambda_1,\lambda_2)$,
where $(\bar u_1,\bar u_2,\bar m_1,\bar m_2)$ is a solution to (\ref{ss1-recalled}). We assume the following regularity and positivity properties for the equilibrium system:
$$
\bar m_i\in C^{2+\theta}(\mathbb T^N),\qquad
\bar u_i\in C^{3+\theta}(\mathbb T^N),\qquad
\bar m_i>0,
\qquad i=1,2.
$$for some $\theta\in(0,1)$. In addition, we assume $f_i\in C^2(\mathcal U)$,
for $i=\{1,2\}$
where $\mathcal U\subset (0,\infty)^2$ is an open neighborhood of the compact set $\{\bar{\mathbf m}(x):x\in\mathbb T^N\}$. Moreover, recall that $(u_1,u_2,m_1,m_2)$ is a solution to (\ref{timedepre}), then we define the perturbations
\begin{align}\label{pertvu}
v_i(t,x)
 =
u_i(t,x)-\bar u_i(x)-\lambda_i(T-t),
\qquad
\mu_i(t,x)
 =
m_i(t,x)-\bar m_i(x),
\end{align}
and the normalized fluctuations
\begin{align}\label{norfluct}
z_i(t,x) =\frac{\mu_i(t,x)}{\bar m_i(x)}.
\end{align}
In view of \eqref{timedepre}, the total masses of $m_1$ and $m_2$ are conserved.  Moreover, noting that $\bar m_i$, $i=1,2$ satisfy the integral conditions given in (\ref{ss1-recalled}), we have
\begin{align}\label{intgconditionmu}
\int_{\mathbb T^N}\mu_i(t,x)\,dx=0,
\qquad i=1,2.
\end{align}

Now, we list some structural hypotheses satisfied by  the ergodic solution $(\bar{ \mathbf u},\bar{\mathbf  m},\boldsymbol{\lambda})$.  Suppose that $\bar m_i$, $i=1,2$ are classical, positive everywhere, which satisfy
\begin{equation}\label{S1}
\nabla\bar m_i=-\bar m_i\nabla\bar u_i,
\qquad i=1,2,
\end{equation}
and the weighted Poincar\'e inequality
\begin{equation}\label{S2}
\int_{\mathbb T^N}
\left(
\varphi-\int_{\mathbb T^N}\varphi \bar m_i\,dx
\right)^2
\bar m_i\,dx
\leq
C_{P,i}
\int_{\mathbb T^N}|\nabla \varphi|^2\bar m_i\,dx,
~\text{for all } \varphi\in W^{1,2}(\bar m_i\,dx), i=1,2.
\end{equation}
We set
\begin{align}\label{matrixA}
A(x)=D_{\mathbf m}\mathbf f(\bar{\mathbf m}(x))
=
\begin{pmatrix}
\partial_{m_1}f_1(\bar{\mathbf m}(x))&
\partial_{m_2}f_1(\bar{\mathbf m}(x))
\\
\partial_{m_1}f_2(\bar{\mathbf m}(x))&
\partial_{m_2}f_2(\bar{\mathbf m}(x))
\end{pmatrix},\quad
D=\begin{pmatrix}
d_1&
0
\\
0&
d_2
\end{pmatrix} 
\end{align}
with $d_1,d_2>0$ and define the weighted symmetric part
$$
DA_s(x) =\frac{DA(x)+A(x)^TD}{2}.
$$
We now assume the following coercivity condition for the two-population system:
\begin{equation}\label{matrixS2}
\int_{\mathbb T^N}
\left[
\boldsymbol\mu^T DA_s(x)\boldsymbol\mu
+
\sum_{i=1}^2d_i
\bar m_i
\left|
\nabla\left(\frac{\mu_i}{\bar m_i}\right)
\right|^2
\right]dx
\geq
\eta
\sum_{i=1}^2d_i
\int_{\mathbb T^N}
\bar m_i
\left|
\nabla\left(\frac{\mu_i}{\bar m_i}\right)
\right|^2dx,
\end{equation}
where $\eta\in(0,1)$ is a constant and $\boldsymbol\mu=(\mu_1,\mu_2)$ satisfies $\mu_i \in L^1(\mathbb T^N)\cap L^\infty(\mathbb T^N)$, $\frac{\mu_i}{\bar m_i}\in W^{1,2}(\bar m_i\,dx)$, $\frac{\mu_i}{\bar m_i}\in L^\infty(\mathbb T^N)$, and
$\displaystyle \int_{\mathbb T^N}\mu_i\,dx=0$ for $i=1,2$.
 We next list several equivalent forms of \eqref{matrixS2} in the following lemma.
\begin{lemma}\label{lemequi}
Let $\bar{\mathbf m}=(\bar m_1,\bar m_2)$ and assume that $\bar m_i$ satisfies a weighted Poincaré inequality.  Suppose also that $A\in L^\infty(\mathbb T^N)$ and constants $d_1,d_2>0$,
then the following properties are equivalent.
\begin{itemize}
\item[(a)] There exists $\eta>0$ such that
\[
\sum_{i=1}^2
\int_{\mathbb T^N}
d_i\bar m_i
\left|
\nabla \left(\frac{\mu_i}{\bar m_i}\right)
\right|^2dx
+
\int_{\mathbb T^N}
\boldsymbol\mu^T DA_s\boldsymbol\mu\,dx
\ge
\eta
\sum_{i=1}^2d_i
\int_{\mathbb T^N}
\bar m_i
\left|
\nabla\left(\frac{\mu_i}{\bar m_i}\right)
\right|^2dx .
\]

\item[(b)] There exists $\eta'>0$ such that
\[
\sum_{i=1}^2
\int_{\mathbb T^N}d_i
\bar m_i
\left|
\nabla \left(\frac{\mu_i}{\bar m_i}\right)
\right|^2dx
+
\int_{\mathbb T^N}
\boldsymbol\mu^TDA_s\boldsymbol\mu\,dx
\ge
\eta'
\sum_{i=1}^2d_i
\int_{\mathbb T^N}
\frac{\mu_i^2}{\bar m_i}\,dx .
\]

\item[(c)] There exists $\eta''>0$ such that
\[
\sum_{i=1}^2
\int_{\mathbb T^N}
d_i\bar m_i
\left|
\nabla \left(\frac{\mu_i}{\bar m_i}\right)
\right|^2dx
+
\int_{\mathbb T^N}
\boldsymbol\mu^T DA_s\boldsymbol\mu\,dx
\ge
\eta''
\sum_{i=1}^2 d_i
\int_{\mathbb T^N}
\mu_i^2\,dx ,
\]
\end{itemize}
for certain constants \(\eta,\eta',\eta''>0\) and for all
\(\mu\in L^1(\mathbb{T}^N)\cap L^\infty(\mathbb{T}^N)\)
such that $ \displaystyle \int_{\mathbb{T}^N} \mu(x)\,dx=0$ and $ \displaystyle\frac{\mu}{\bar m}\in W^{1,2}(\bar m\,dx)$.
\end{lemma}

\begin{proof}

We first prove that $(a)$ implies $(b)$. Since
\[
\int_{\mathbb T^N} \frac{\mu_i}{\bar m_i}\bar m_i\,dx
=
\int_{\mathbb T^N}\mu_i\,dx
=0,
\]
the weighted Poincaré inequality gives
\[
\int_{\mathbb T^N} \frac{\mu_i^2}{\bar m_i}\,dx
=
\int_{\mathbb T^N}
\left(\frac{\mu_i}{\bar m_i}\right)^2
\bar m_i\,dx
\le
C_{P,i}
\int_{\mathbb T^N}
\bar m_i
\left|
\nabla\left(\frac{\mu_i}{\bar m_i}\right)
\right|^2dx .
\]
Let $C_P = \max\{C_{P,1},C_{P,2}\}$. Then we have 
\[
\sum_{i=1}^2d_i
\int_{\mathbb T^N}
\frac{\mu_i^2}{\bar m_i}\,dx
\le
C_P
\sum_{i=1}^2d_i
\int_{\mathbb T^N}
\bar m_i
\left|
\nabla\!\left(\frac{\mu_i}{\bar m_i}\right)
\right|^2dx ,
\]
where $d_1,d_2>0.$
By applying $(a)$, we deduce that 
\[
\sum_{i=1}^2
\int_{\mathbb T^N}d_i
\bar m_i
\left|
\nabla \left(\frac{\mu_i}{\bar m_i}\right)
\right|^2dx
+
\int_{\mathbb T^N}
\boldsymbol\mu^T DA_s\boldsymbol\mu\,dx
\ge
\eta
\sum_{i=1}^2d_i
\int_{\mathbb T^N}
\bar m_i
\left|
\nabla \left(\frac{\mu_i}{\bar m_i}\right)
\right|^2dx
\ge
\frac{\eta}{C_P}
\sum_{i=1}^2d_i
\int_{\mathbb T^N}
\frac{\mu_i^2}{\bar m_i}\,dx .
\]
Therefore, $(b)$ is satisfied with $\eta'=\frac{\eta}{C_P}$. We now show that $(b)$ implies $(c)$. Since $\bar m_i\in L^\infty(\mathbb T^N)$, we have
\[
\int_{\mathbb T^N}
\frac{\mu_i^2}{\bar m_i}\,dx
\ge
\frac{1}{\|\bar m_i\|_{L^\infty(\mathbb T^N)}}
\int_{\mathbb T^N}
\mu_i^2\,dx .
\]
We define $\displaystyle
M =\max_{i=1,2}\|\bar m_i\|_{L^\infty(\mathbb T^N)}$. Then, we have 
\[
\sum_{i=1}^2d_i
\int_{\mathbb T^N}
\frac{\mu_i^2}{\bar m_i}\,dx
\ge
\frac{1}{M}
\sum_{i=1}^2
d_i\int_{\mathbb T^N}
\mu_i^2\,dx ,
\]
where $d_1,d_2>0$ are constants.
Moreover, $(b)$ yields
\[
\sum_{i=1}^2
\int_{\mathbb T^N}
d_i\bar m_i
\left|
\nabla \left(\frac{\mu_i}{\bar m_i}\right)
\right|^2dx
+
\int_{\mathbb T^N}
\boldsymbol\mu^T DA_s\boldsymbol\mu\,dx
\ge
\frac{\eta'}{M}
\sum_{i=1}^2d_i
\int_{\mathbb T^N}
\mu_i^2\,dx ,
\]
which is the conclusion $(c)$ with $\displaystyle
\eta''=\frac{\eta'}{M}$. Finally, we verify that $(c)$ implies $(a)$. Fix $\varepsilon\in(0,1)$. We write
\begin{align}
\int_{\mathbb T^N}\boldsymbol\mu^T DA_s(x)\boldsymbol\mu
+
\sum_{i=1}^2
d_i\bar m_i
\left|
\nabla\left(\frac{\mu_i}{\bar m_i}\right)
\right|^2
 dx
=&
(1-\varepsilon)\int_{\mathbb T^N}\bigg[\boldsymbol\mu^T DA_s(x)\boldsymbol\mu
+
\sum_{i=1}^2d_i
\bar m_i
\left|
\nabla\left(\frac{\mu_i}{\bar m_i}\right)
\right|^2
\bigg]dx\nonumber\\
+&
\varepsilon \int_{\mathbb T^N}\bigg[\boldsymbol\mu^T DA_s(x)\boldsymbol\mu
+
\sum_{i=1}^2d_i
\bar m_i
\left|
\nabla\left(\frac{\mu_i}{\bar m_i}\right)
\right|^2
\bigg]dx .
\end{align}
By applying $(c)$ to the first term, we get
\begin{equation*}
\begin{aligned}
\int_{\mathbb T^N}\boldsymbol\mu^TD A_s(x)\boldsymbol\mu
+&
\sum_{i=1}^2d_i
\bar m_i
\left|
\nabla\left(\frac{\mu_i}{\bar m_i}\right)
\right|^2
 dx\\
\ge &
(1-\varepsilon)\eta''
\sum_{i=1}^2
\int_{\mathbb T^N} d_i\mu_i^2\,dx
+
\varepsilon
\sum_{i=1}^2
\int_{\mathbb T^N}d_i
\bar m_i
\left|
\nabla \left(\frac{\mu_i}{\bar m_i}\right)
\right|^2dx
+
\varepsilon
\int_{\mathbb T^N}
\boldsymbol\mu^T DA_s(x)\boldsymbol\mu\,dx .
\end{aligned}
\end{equation*}
Since $A\in L^\infty$, there exists $K=\|A\|_{L^\infty(\mathbb T^N)}$ such that
\begin{align}
&\int_{\mathbb T^N}\boldsymbol\mu^T DA_s(x)\boldsymbol\mu
+
\sum_{i=1}^2d_i
\bar m_i
\left|
\nabla\left(\frac{\mu_i}{\bar m_i}\right)
\right|^2
 dx\nonumber\\
\ge&
\left((1-\varepsilon)\eta''-\varepsilon \sqrt{\frac{d_{\max}}{d_{\min}}}K\right)
\sum_{i=1}^2
\int_{\mathbb T^N} d_i\mu_i^2\,dx
+
\varepsilon
\sum_{i=1}^2
\int_{\mathbb T^N}d_i
\bar m_i
\left|
\nabla \left(\frac{\mu_i}{\bar m_i}\right)
\right|^2dx ,
\end{align}
where 
\begin{align}\label{dmaxdmin}
d_{\max}=\max\{d_1,d_2\} ,~~d_{\min}=\min\{d_1,d_2\}.
\end{align}
We choose $ \displaystyle \varepsilon = \frac{\eta''}{\eta''+\sqrt{\frac{d_{\max}}{d_{\min}}}K}$ so that $( \displaystyle 1-\varepsilon)\eta''-\varepsilon \sqrt{\frac{d_{\max}}{d_{\min}}}K=0$. Then, we get 
\[
\int_{\mathbb T^N}\boldsymbol\mu^T DA_s(x)\boldsymbol\mu
+
\sum_{i=1}^2d_i
\bar m_i
\left|
\nabla\left(\frac{\mu_i}{\bar m_i}\right)
\right|^2
 dx
\ge
\frac{\eta''}{\eta''+K}
\sum_{i=1}^2d_i
\int_{\mathbb T^N}
\bar m_i
\left|
\nabla\left(\frac{\mu_i}{\bar m_i}\right)
\right|^2dx .
\]
Thus, $(a)$ follows with $\displaystyle \eta=\frac{\eta''}{\eta''+\sqrt{\frac{d_{\max}}{d_{\min}}}\|A\|_{L^\infty(\mathbb T^N)}}$. This concludes the proof.\end{proof}
 
On the basis of the preliminary definitions and assumptions introduced above, we are now prepared to state the main result concerning the local turnpike stability of the stationary equilibrium $(\bar{\mathbf u}, \bar{\mathbf m})$.

\section{Matrix Coercivity and Local Turnpike Stability}\label{sectwopopturnpike}

Throughout this section, we restrict attention to classical solutions. We first examine the perturbation system satisfied by \((v_i,\mu_i)_{i=1,2}\). By combining \eqref{timedepre} with \eqref{ss1-recalled}, we deduce that the pair \((v_i,\mu_i)_{i=1,2}\), as defined in \eqref{pertvu}, satisfies the following system:
\begin{align}\label{pertsystem}
\left\{
\begin{array}{ll}
-\partial_t v_i-\Delta v_i+\nabla\bar u_i\cdot\nabla v_i
=
f_i(\bar{\mathbf m}+\boldsymbol\mu)
-
f_i(\bar{\mathbf m})
-\dfrac12|\nabla v_i|^2, &x\in\mathbb T^N,\; t\in (0,T),
\\[0.4em]
\partial_t\mu_i-\Delta\mu_i
-\nabla\cdot(\mu_i\nabla\bar u_i)
-\nabla\cdot(\bar m_i\nabla v_i)
=
\nabla\cdot(\mu_i\nabla v_i),&x\in\mathbb T^N,\; t\in (0,T),\\[0.4em]
v_i(T,x)
=
u_{iT}(x)-\bar u_i(x),&x\in\mathbb T^N, 
\\[0.2em]
\mu_i(0,x)
=
m_{i0}(x)-\bar m_i(x),&x\in\mathbb T^N,  
\end{array}
\right.
\end{align}
for $i=1,2$. Using \eqref{S1} and \eqref{norfluct}, the system \eqref{pertsystem} can be equivalently expressed as
\begin{align}\label{weightpertsystem}
\left\{
\begin{array}{ll}
-\partial_t v_i
-\dfrac1{\bar m_i}\nabla\cdot(\bar m_i\nabla v_i)
 =
f_i(\bar{\mathbf m}+\boldsymbol\mu)
-
f_i(\bar{\mathbf m})
-\dfrac12|\nabla v_i|^2,&x\in\mathbb T^N,\; t\in (0,T),
\\[0.4em]
\partial_t\mu_i
-\nabla\cdot(\bar m_i\nabla z_i)
-\nabla\cdot(\bar m_i\nabla v_i)
=
\nabla\cdot(\mu_i\nabla v_i),&x\in\mathbb T^N,\; t\in (0,T),\\[0.4em]
v_i(T,x)
=
u_{iT}(x)-\bar u_i(x),&x\in\mathbb T^N,
\\[0.2em]
z_i(0,x)
=
\dfrac{m_{i0}}{\bar m_i}(x)-1,&x\in\mathbb T^N.
\end{array}
\right.
\end{align}
We aim to investigate the dynamical behavior of $(v_i,z_i)_{i=1,2}$ given in (\ref{weightpertsystem}). As shown in \cite{cirant2025local}, we first construct a Lyapunov functional associated with system (\ref{weightpertsystem}). In this work, we exploit the turnpike property satisfied by the Lyapunov functional and subsequently derive the desired decay estimates for $(v_i,\mu_i)_{i=1,2}$ under an appropriate local smallness assumption. Finally, the existence of at least one trajectory satisfying this smallness condition is obtained by means of a standard Leray–Schauder fixed-point argument.  To this end, we introduce the functional
\[
\Phi_D(t)
 =
\sum_{i=1}^2d_i
\int_{\mathbb T^N}
\left(
\mu_i v_i+\frac{\mu_i^2}{\bar m_i}
\right)\,dx
\]
with $d_1,d_2>0$ and define
\begin{align}\label{funcGK}
G_D(t)
 &=
\sum_{i=1}^2
d_i\int_{\mathbb T^N}\bar m_i|\nabla z_i|^2\,dx
=\sum_{i=1}^2G_i,
\qquad
K_D(t)
 =
\sum_{i=1}^2d_i
\int_{\mathbb T^N}\bar m_i|\nabla v_i|^2\,dx
=\sum_{i=1}^2K_i,
\end{align}
as well as
\begin{align}\label{funcZ}
Z_D(t)=\sum_{i=1}^2
d_i\int_{\mathbb T^N}\frac{\mu_i^2}{\bar m_i}\,dx = \sum_{i=1}^2d_i
\int_{\mathbb T^N}z_i^2\bar m_i\,dx
=\sum_{i=1}^2Z_i,
\end{align}
where \((v_i,\mu_i,z_i)_{i=1,2}\) denotes a solution to \eqref{weightpertsystem}.  
We begin by deriving the following a priori estimate for \(\Phi_D\).

\begin{lemma}\label{lemPhicontrol}
There is a constant $C_\Phi>0$, depending only on $\bar m_1,\bar m_2$ and on the weighted Poincaré constants \eqref{S2}, such that
$$
|\Phi_D(t)| \leq C_\Phi\bigl(G_D(t)+K_D(t)\bigr) \qquad \forall t\in[0,T].
$$ 
\end{lemma}
\begin{proof}
Using the condition $\displaystyle \int_{\mathbb T^N}\mu_i\,dx=0$ from \eqref{intgconditionmu}, we can write
$$
\int_{\mathbb T^N}\mu_i v_i\,dx
=
\int_{\mathbb T^N}
\mu_i
\left(
v_i-\int_{\mathbb T^N}v_i\bar m_i\,dx
\right)dx.
$$
Next, by applying the Cauchy-Schwarz inequality together with \eqref{S2}, we obtain
$$
\left|
d_i\int_{\mathbb T^N}\mu_i v_i\,dx
\right|
\leq
\left(
d_i\int_{\mathbb T^N}\frac{\mu_i^2}{\bar m_i}\,dx
\right)^{1/2}
\left(d_i
\int_{\mathbb T^N}
\left(
v_i-\int_{\mathbb T^N}v_i\bar m_i\,dx
\right)^2
\bar m_i\,dx
\right)^{1/2}
\leq
C Z_i(t)^{1/2}K_i(t)^{1/2}.
$$
It follows that
$$
\left|
d_i\int_{\mathbb T^N}\mu_i v_i\,dx
\right|
\leq
C\bigl(Z_i(t)+K_i(t)\bigr).
$$
In addition, since $ \displaystyle\int_{\mathbb T^N}\mu_i\,dx=0$, using \eqref{S2} with
$z_i=\mu_i/\bar m_i$ gives
$$
Z_i(t)\leq C_{P,i}G_i(t).
$$
Summing these estimates for $i=1,2$ yields the desired result and thus completes the proof of the Lemma.
\end{proof}
\begin{lemma}\label{lemdissipationid}
Assume that \eqref{S1} is satisfied. Let
$(v_i,\mu_i)_{i=1,2}$
be a classical solution of \eqref{weightpertsystem}. For $t\in[0,T]$, we define
\begin{align}\label{Ridef}
R_i(x,t)
 =
f_i(\bar{\mathbf m}+\boldsymbol\mu)
-
f_i(\bar{\mathbf m})
-
\sum_{j=1}^2 A_{ij}(x)\mu_j.
\end{align}
Then, we have 
\begin{equation}\label{exactPhiid}
-\frac{d}{dt}\Phi_D(t) = \mathcal B_D[\boldsymbol\mu(t)]
+\sum_{i=1}^2\int_{\mathbb T^N}d_i \bar m_i|\nabla v_i+\nabla z_i|^2\,dx +\mathcal R_D(t),
\end{equation}
where
\begin{align}\label{Betamu}
\mathcal B_D[\boldsymbol\mu]
=\int_{\mathbb T^N} \left[ \boldsymbol\mu^TDA_s(x)\boldsymbol\mu
+\sum_{i=1}^2 d_i\bar m_i|\nabla z_i|^2\right]dx,
\end{align}
and
$$
\mathcal R_D(t)
 =
\sum_{i=1}^2d_i
\int_{\mathbb T^N}
\left[
\mu_iR_i
+
\frac12\mu_i|\nabla v_i|^2
+
2\mu_i\nabla v_i\cdot\nabla z_i
\right]dx.
$$
Moreover, there exist $\Theta_0>0$ and $c_0>0$ such that, if
\begin{align}\label{mu_nabla_v_small}
\sup_{t\in[0,T]}
\sum_{i=1}^2
\left(
\|\mu_i(t)\|_{L^\infty}
+
\|\nabla v_i(t)\|_{L^\infty}
\right)
\leq \Theta_0,
\end{align}
thus we deduce
\begin{equation}\label{coerPhidiff}
-\frac{d}{dt}\Phi_D(t)
\geq
c_{0,D}\bigl(G_D(t)+K_D(t)\bigr)
\qquad
\text{for every }t\in[0,T],
\end{equation}
where $c_{0,D}>0$ is a constant depending on $d_1, d_2>0$.
\end{lemma}
\begin{proof}
By using \eqref{Ridef} and \eqref{matrixA}, we rewrite
 \eqref{weightpertsystem} as
\begin{align}
\left\{\begin{array}{ll}
\partial_t\mu_i
=
\nabla\cdot(\bar m_i\nabla z_i)
+
\nabla\cdot(\bar m_i\nabla v_i)
+
\nabla\cdot(\mu_i\nabla v_i),&x\in\mathbb T^N,t\in(0,T),\\
\partial_t v_i
=
-\frac1{\bar m_i}\nabla\cdot(\bar m_i\nabla v_i)
-
\sum_{j=1}^2A_{ij}\mu_j
-
R_i
+
\frac12|\nabla v_i|^2,&x\in\mathbb T^N,t\in(0,T).
\end{array}
\right.
\end{align}
We calculate the dynamics of $\displaystyle d_i\int_{\mathbb T^N}\mu_i v_i\,dx$ and use the integration by parts to get
$$
\frac{d}{dt}
\int_{\mathbb T^N}d_i\mu_i v_i\,dx
=
-\int_{\mathbb T^N}d_i\bar m_i|\nabla v_i|^2\,dx
-
\frac {d_i}{2}
\int_{\mathbb T^N}\mu_i|\nabla v_i|^2\,dx
-d_i
\int_{\mathbb T^N}
\mu_i
\left(
\sum_{j=1}^2A_{ij}\mu_j+R_i
\right)dx.
$$
Similarly, we have 
\begin{align*}
\frac{d}{dt}d_i
\int_{\mathbb T^N}\frac{\mu_i^2}{\bar m_i}\,dx
=&
2d_i\int_{\mathbb T^N}z_i\partial_t\mu_i\,dx\\
=&-2d_i\int_{\mathbb T^N}\bar m_i|\nabla z_i|^2\,dx
-
2d_i\int_{\mathbb T^N}\bar m_i\nabla z_i\cdot\nabla v_i\,dx
-
2d_i\int_{\mathbb T^N}\mu_i\nabla z_i\cdot\nabla v_i\,dx.
\end{align*}
By adding the two identities above and summing over $i=1,2$, we obtain
$$
-\frac{d}{dt}\Phi_D(t)
=
\int_{\mathbb T^N}\boldsymbol\mu^TDA_s\boldsymbol\mu\,dx
+
\sum_{i=1}^2
\int_{\mathbb T^N}d_i
\left[
\bar m_i|\nabla v_i|^2
+
2\bar m_i\nabla v_i\cdot\nabla z_i
+
2\bar m_i|\nabla z_i|^2
\right]dx
+
\mathcal R_D(t).
$$
Since we have 
$$
\bar m_i|\nabla v_i|^2
+
2\bar m_i\nabla v_i\cdot\nabla z_i
+
2\bar m_i|\nabla z_i|^2
=
\bar m_i|\nabla v_i+\nabla z_i|^2
+
\bar m_i|\nabla z_i|^2,
$$
which implies \eqref{exactPhiid} holds.

It remains to estimate $\mathcal R_D(t)$. Since $\mathbf f\in C^2$ near
$\bar{\mathbf m}$, Taylor's formula gives
$$
|d_iR_i(x,t)|\leq C_fd_i|\boldsymbol\mu(x,t)|^2
$$
whenever $\boldsymbol\mu$ is sufficiently small. Since $\bar m_i$ is smooth and strictly positive on $\mathbb T^N$, we have from Poincar\'e inequality that 
$$
\int_{\mathbb T^N}d_i|\boldsymbol\mu|^2\,dx
\leq
\frac{d_{\max}C}{d_{\min}} Z_D(t)
\leq
\frac{d_{\max}C}{d_{\min}}G_D(t),
$$
where $C>0$ is a constant and $d_{\max}$, $d_{\min}$ are given in \eqref{dmaxdmin}.
Thus, by using (\ref{mu_nabla_v_small}), we find there exists constant $C>0$ such that 
$$
\left|
\sum_{i=1}^2d_i
\int_{\mathbb T^N}\mu_iR_i\,dx
\right|
\leq
\frac{Cd_{\max}}{d_{\min}}\|\boldsymbol\mu\|_{L^\infty}Z_D(t)
\leq
\frac{Cd_{\max}}{d_{\min}}\Theta_0G_D(t).
$$
Similarly, thanks to (\ref{mu_nabla_v_small}) and the boundedness of $\mu_i$, $i=1,2$, we have for some constant $C>0,$
$$
\left|
\sum_{i=1}^2
\int_{\mathbb T^N}\frac{d_i}{2}\mu_i|\nabla v_i|^2\,dx
\right|
\leq
C\Theta_0 K_D(t),
$$
and
$$
\left|
\sum_{i=1}^2
d_i\int_{\mathbb T^N}2\mu_i\nabla v_i\cdot\nabla z_i\,dx
\right|
\leq
C\Theta_0\bigl(K_D(t)+G_D(t)\bigr),
$$
where we have used Young's inequality.
Hence
$$
|\mathcal R_D(t)|
\leq
C_D\Theta_0\bigl(K_D(t)+G_D(t)\bigr),
$$
where constant $C_D>0$ depends on $d_{1}$ and $d_2.$
By using the coercive assumption \eqref{matrixS2}, we find
$$
\mathcal B_D[\boldsymbol\mu(t)]\geq \eta G_D(t),
$$
where $\mathcal B_D$ and $G_D$ are given in (\ref{Betamu}) and (\ref{funcGK}), respectively.
Moreover, we have the fact that for every $\eta\in(0,1)$, there exists $c_\eta>0$ such that
$$
|\xi+\zeta|^2+\eta|\zeta|^2
\geq
c_\eta\bigl(|\xi|^2+|\zeta|^2\bigr)
\qquad
\forall \xi,\zeta\in\mathbb R^N.
$$
By using the previous inequality,  we obtain
$$
\mathcal B_D[\boldsymbol\mu(t)]
+
\sum_{i=1}^2 d_i
\int_{\mathbb T^N}
\bar m_i|\nabla v_i+\nabla z_i|^2\,dx
\geq
c_\eta\bigl(G_D(t)+K_D(t)\bigr).
$$
Choosing $\Theta_0>0$ sufficiently small gives \eqref{coerPhidiff}. This finishes the proof. 
\end{proof}

\begin{proposition}
\label{prop:Phi-exponential}
Assume all assumptions in Lemma~\ref{lemdissipationid} hold. Then there exists
$\sigma_D>0$, independent of $T$, such that
\begin{align}\label{turnpike_decay_phi}
\Phi_D(T)e^{-\sigma_D(T-t)}
\leq
\Phi_D(t)
\leq
\Phi_D(0)e^{-\sigma_D t},
\qquad
0\leq t\leq T,
\end{align}
where $\sigma_D>0$ is a constant depending on $d_1, d_2>0$.
Moreover, for every $0\leq s\leq t\leq T$,
$$
\int_s^t\bigl(G_D(\tau)+K_D(\tau)\bigr)\,d\tau
\leq
C_D
\left(
|\Phi_D(0)|e^{-\sigma_D s}
+
|\Phi_D(T)|e^{-\sigma_D(T-t)}
\right),
$$
where constant $C_D>0$ depends on $d_1,d_2>0.$
\end{proposition}

\begin{proof}
By Lemma~\ref{lemPhicontrol} and \eqref{coerPhidiff}, we have
$$
\frac{d}{dt}\Phi_D(t)
\leq
-\sigma_D |\Phi_D(t)|
$$
for some $\sigma_D>0$ depending on $d_1$ and $d_2$. Consequently, we obtain
$$
\frac{d}{dt}\bigl(e^{\sigma_D t}\Phi_D(t)\bigr)
=
e^{\sigma_D t}\bigl(\Phi'_D(t)+\sigma_D\Phi_D(t)\bigr)
\leq 0,
$$
which yields the exponential upper bound $\displaystyle \Phi_D(t)\leq \Phi_D(0)e^{-\sigma_D t}$.
In an analogous way, we have 
$$
\frac{d}{dt}\bigl(e^{-\sigma_D t}\Phi_D(t)\bigr)
=
e^{-\sigma_D t}\bigl(\Phi_D'(t)-\sigma_D\Phi_D(t)\bigr)
\leq 0,
$$
and hence
$$
\Phi_D(t)\geq \Phi_D(T)e^{-\sigma_D(T-t)}.
$$
Finally, integrating \eqref{coerPhidiff} over the interval $[s,t]$ leads to
$$
c_{0,D}\int_s^t (G_D+K_D)\,d\tau
\leq
\Phi_D(s)-\Phi_D(t).
$$
Combining this with the previously established upper and lower bounds for $\Phi_D$ yields the required integral estimate. This concludes the proof. 
\end{proof}

We now derive the required weighted estimates for $\mu_i$ and $\nabla v_i$ for $i = 1, 2$ as follows:


\begin{proposition}\label{prop:L2-turnpike}
Assume all assumptions in Proposition~\ref{prop:Phi-exponential} hold. Then there exists a constant $C_D>0$ depending on $d_1$ and $d_2$, but independent of $T$, such that
$$
Z_D(t)+K_D(t)
\leq
C_D
\left[
\bigl(Z_D(0)+|\Phi_D(0)|\bigr)e^{-\sigma_D t}
+
\bigl(K_D(T)+|\Phi_D(T)|\bigr)e^{-\sigma_D(T-t)}
\right]
$$
for every $t\in[0,T]$.
\end{proposition}

\begin{proof}
We first estimate $Z(t)$ and use the equation of $\mu_i$ given in (\ref{weightpertsystem}) to obtain
$$
\frac{d_i}2\frac{d}{dt}
\int_{\mathbb T^N}\frac{\mu_i^2}{\bar m_i}\,dx
=
-d_i\int_{\mathbb T^N}\bar m_i|\nabla z_i|^2\,dx
-d_i
\int_{\mathbb T^N}\bar m_i\nabla v_i\cdot\nabla z_i\,dx
-d_i
\int_{\mathbb T^N}\mu_i\nabla v_i\cdot\nabla z_i\,dx.
$$
With this, using Young's inequality and the smallness assumption of Lemma~\ref{lemdissipationid}, we further obtain that for some constant $C>0,$
\begin{align}\label{differential_inequality_ZG}
\frac{d}{dt}Z_D(t)+G_D(t)\leq C K_D(t),
\end{align}
where $Z_D$, $G_D$ and $K_D$ are given in (\ref{funcZ}) and (\ref{funcGK}), respectively.
Let $t\in[1,T]$, then we choose $\tau\in[t-1,t]$ such that
$$
Z_D(\tau)\leq \int_{t-1}^t Z_D(s)\,ds.
$$
By using the weighted Poincaré inequality, we have $Z_D\leq C_P G_D$ for some constant $C_P>0$, and therefore
$$
Z_D(\tau)\leq C_P\int_{t-1}^tG_D(s)\,ds.
$$
Integrating the differential inequality (\ref{differential_inequality_ZG}) from $\tau$ to $t$ yields
$$
Z_D(t)\leq C\int_{t-1}^t\bigl(G_D(s)+K_D(s)\bigr)\,ds,
$$
for some constant $C>0.$
Applying Proposition~\ref{prop:Phi-exponential}, we consequently obtain
$$
Z_D(t)
\leq
C_D
\left(
|\Phi_D(0)|e^{-\sigma_D t}
+
|\Phi_D(T)|e^{-\sigma_D(T-t)}
\right),
$$
for constant $C_D>0$ depending on $d_1$ and $d_2.$
Then for $t\in[0,1]$, we get 
$$
Z_D(t)
\leq
Z_D(0)+C\int_0^tK_D(s)\,ds,
$$
and, in combination with the integral estimate provided by Proposition~\ref{prop:Phi-exponential}, this yields the same bound, now with the additional contribution $Z_D(0)e^{-\sigma_D t}$.

We now derive an estimate for $K_D(t)$. We differentiate $K_i$ given in (\ref{funcGK}) over time, and using the fact that $\bar m_i$ is independent of $t$, we obtain
\[
-\frac12\frac{d}{dt}K_i(t)
=
d_i\int_{\mathbb T^N}
\bar m_i\nabla v_i\cdot\nabla(-\partial_t v_i)\,dx.
\]
Invoking the $v_i$-equation in (\ref{weightpertsystem}),
we deduce
\[
\begin{aligned}
-\frac12\frac{d}{dt}K_i(t)
&=
d_i\int_{\mathbb T^N}
\bar m_i\nabla v_i\cdot\nabla\Delta v_i\,dx
-d_i
\int_{\mathbb T^N}
\bar m_i\nabla v_i\cdot
\nabla(\nabla\bar u_i\cdot\nabla v_i)\,dx
\\
&\quad
+d_i
\int_{\mathbb T^N}
\bar m_i\nabla v_i\cdot\nabla [f_i(\bar{\mathbf m}+\boldsymbol\mu)
-
f_i(\bar{\mathbf m})]\,dx
-
\frac{d_i}{2}
\int_{\mathbb T^N}
\bar m_i\nabla v_i\cdot
\nabla |\nabla v_i|^2\,dx\\
:&=II_1+II_2+II_3+II_4.
\end{aligned}
\]
Then, we apply integration by parts to obtain
\[
II_1=d_i\int_{\mathbb T^N}
\bar m_i\nabla v_i\cdot\nabla\Delta v_i\,dx
=
-d_i\int_{\mathbb T^N}
\bar m_i|\Delta v_i|^2\,dx
-d_i
\int_{\mathbb T^N}
\nabla\bar m_i\cdot\nabla v_i\,\Delta v_i\,dx.
\]
By using  Young's inequality, one has
\begin{align}\label{II1}
\vert II_1\vert
\le
-\frac{d_i}{2}
\int_{\mathbb T^N}
\bar m_i|\Delta v_i|^2\,dx
+
C K_i(t).
\end{align}
Similarly, since $\bar u_i$ is classical, we get 
\begin{align}\label{II2}
\left|
II_2
\right|=&d_i\bigg|
\int_{\mathbb T^N}
\bar m_i
\nabla v_i\cdot
(\nabla^2\bar u_i)\nabla v_i\,dx
+
\int_{\mathbb T^N}
\bar m_i
\nabla v_i\cdot
(\nabla^2 v_i)\nabla\bar u_i\,dx
\bigg|\nonumber\\
\le &
\frac{d_i}{4}
\int_{\mathbb T^N}
\bar m_i|\Delta v_i|^2\,dx
+
C K_i(t).
\end{align}
Noting that 
\[
f_i(\bar{\mathbf m}+\boldsymbol\mu)
-
f_i(\bar{\mathbf m})
=
\sum_{j=1}^2 A_{ij}\mu_j+R_i,
\qquad
|R_i|
\leq
C|\boldsymbol\mu|^2.
\]
Since
\[
\begin{aligned}
d_i\int_{\mathbb T^N}
\bar m_i\nabla v_i\cdot\nabla(A_{ij}\mu_j)\,dx
&=
-d_i\int_{\mathbb T^N}
A_{ij}\mu_j
\nabla\cdot(\bar m_i\nabla v_i)\,dx,
\end{aligned}
\]
we use Young's inequality, $f_i\in C^2$ and the boundedness of $\bar m_i$, $i=1,2$ to obtain 
\begin{equation}\label{3p13II2_1}
\begin{aligned}
\left|
d_i\int_{\mathbb T^N}
\bar m_i\nabla v_i\cdot\nabla(A_{ij}\mu_j)\,dx
\right|
&\leq
Cd_i
\int_{\mathbb T^N}
|\mu_j|
\,|\nabla\cdot(\bar m_i\nabla v_i)|\,dx
\\
&\leq
Cd_i
\int_{\mathbb T^N}
|\mu_j|
\bigl(
|\Delta v_i|
+
|\nabla v_i|
\bigr)\,dx
\\
&\leq
\varepsilon
d_i\int_{\mathbb T^N}
\bar m_i|\Delta v_i|^2\,dx
+
C_\varepsilon d_i
\int_{\mathbb T^N}
|\mu_j|^2\,dx
+
C K_i(t)
\\
&\leq
\varepsilon
d_i\int_{\mathbb T^N}
\bar m_i|\Delta v_i|^2\,dx
+
C_\varepsilon d_{\max} Z_D(t)
+
C K_i(t),
\end{aligned}
\end{equation}
where $\varepsilon>0$ is small enough and $C_{\varepsilon}>0$ is a constant depending on $\varepsilon$.
In addition,
\[
\begin{aligned}
\int_{\mathbb T^N}
\bar m_i\nabla v_i\cdot\nabla R_i\,dx
&=
-\int_{\mathbb T^N}
R_i
\nabla\cdot(\bar m_i\nabla v_i)\,dx .
\end{aligned}
\]
Since
\[
|R_i|
\leq
C|\boldsymbol\mu|^2
\leq
C\Theta_0|\boldsymbol\mu|,
\]
it follows from Young's inequality and (\ref{mu_nabla_v_small}) that
\begin{equation}\label{II2_2}
\begin{aligned}
\left|
d_i\int_{\mathbb T^N}
\bar m_i\nabla v_i\cdot\nabla R_i\,dx
\right|
&\leq
C\Theta_0 d_i
\int_{\mathbb T^N}
|\boldsymbol\mu|
\,|\nabla\cdot(\bar m_i\nabla v_i)|\,dx
\\
&\leq
C\Theta_0 d_i
\int_{\mathbb T^N}
|\boldsymbol\mu|
\bigl(
|\Delta v_i|
+
|\nabla v_i|
\bigr)\,dx
\\
&\leq
\varepsilon
d_i\int_{\mathbb T^N}
\bar m_i|\Delta v_i|^2\,dx
+
C_\varepsilon d_{\max}Z_D(t)
+
C K_i(t),
\end{aligned}
\end{equation}
where $\varepsilon>0$ is small enough and $C_{\varepsilon}>0$ is a constant depending on $\varepsilon$.
Combining (\ref{3p13II2_1}) with (\ref{II2_2}), we obtain
\begin{align}\label{II22}
\vert II_2\vert\leq\varepsilon d_i\int_{\mathbb T^N}\bar m_i|\Delta v_i|^2\,dx+ C_{\varepsilon }d_{\max}Z_D+CK_i.
\end{align}
 
Similarly, since $\nabla|\nabla v_i|^2
=
2(\nabla^2 v_i)\nabla v_i,$
we have from Young's inequality that
\begin{equation}\label{315II3}
\begin{aligned}
|\frac{d_i}{2}
\int_{\mathbb T^N}
\bar m_i\nabla v_i\cdot\nabla|\nabla v_i|^2\,dx|
&\le
C d_i
\int_{\mathbb T^N}
\bar m_i|\nabla v_i|^2|\nabla^2 v_i|\,dx
\\
&\le
C d_i\|\nabla v_i\|_{L^\infty}
\int_{\mathbb T^N}
\bar m_i|\nabla v_i||\nabla^2 v_i|\,dx
\\
&\le
\varepsilon d_i
\int_{\mathbb T^N}
\bar m_i|\nabla^2 v_i|^2\,dx
+
C_\varepsilon
K_i(t),
\end{aligned}
\end{equation}
where $C>0$ is a constant, $\varepsilon>0$ is small enough and $C_{\varepsilon}>0 $ is a constant depending on $\varepsilon.$  Then, by using the boundedness of $\bar m_i$, one has
\begin{align}\label{Hessian_estimate_vi}
\int_{\mathbb T^N}
\bar m_i|\nabla^2 v_i|^2\,dx
\le
C
\int_{\mathbb T^N}
\bar m_i|\Delta v_i|^2\,dx.
\end{align}
Substituting (\ref{Hessian_estimate_vi}) into (\ref{315II3}), one gets
\begin{align}\label{II3}
|\frac{d_i}{2}
\int_{\mathbb T^N}
\bar m_i\nabla v_i\cdot\nabla|\nabla v_i|^2\,dx|
\le
\varepsilon d_i
\int_{\mathbb T^N}
\bar m_i|\Delta v_i|^2\,dx
+
C_\varepsilon K_i(t).
\end{align}
We collect (\ref{II1}), (\ref{II22}) and (\ref{II3}) to obtain
\begin{align}\label{ODI_K}
-\frac{d}{dt}K_D(t)
\leq
C_D\bigl(K_D(t)+Z_D(t)\bigr).
\end{align}
Here, the constant $C$ depends on $d_1,$ $d_2$, $\bar{\mathbf m}$, $\bar{\mathbf u}$, the $C^2$-norm of $\mathbf f$ in a neighborhood of $\bar{\mathbf m}$, and on the smallness parameter $\Theta_0$, but is independent of $T$. 
 In addition, fix $t\in[0,T-1]$ and  select $\tau\in[t,t+1]$ such that
$$
K_D(\tau)\leq \int_t^{t+1}K_D(s)\,ds.
$$
Integrating the above differential inequality from $t$ to $\tau$, exploiting the previously obtained estimate for $Z$, and then applying Proposition~\ref{prop:Phi-exponential}, we obtain
$$
K_D(t)
\leq
C_D
\left(
|\Phi_D(0)|e^{-\sigma_D t}
+
|\Phi_D(T)|e^{-\sigma_D(T-t)}
\right),
$$
where $C_D>0$ is a constant depending on $d_1$ and $d_2.$ 
For $t\in[T-1,T]$, integrating (\ref{ODI_K}) instead from $t$ to $T$ gives
$$
K_D(t)
\leq
C_D
\left(
|\Phi_D(0)|e^{-\sigma_D t}
+
|\Phi_D(T)|e^{-\sigma_D(T-t)}
+
K_D(T)e^{-\sigma_D(T-t)}
\right).
$$
Collecting the above estimates for $Z_D$ and $K_D$ completes the proof of the proposition.
\end{proof}

Then, we establish the $L^\infty$ estimate of $\mu_i$, $i=1,2$ in the following lemma.  
\begin{lemma}\label{lembootstrap}
Let the hypotheses of Proposition~\ref{prop:L2-turnpike} hold. Assume in addition that the solution is uniformly bounded in
$$
C^{1+\alpha/2,2+\alpha}([0,T]\times\mathbb T^N)
$$
by a constant independent of $T$. Then, possibly decreasing $\sigma$, there exists $C_D>0$ depending on $d_1,d_2>0$ such that
\begin{align}\label{mui_Linf_estimate}
&\sum_{i=1}^2
\|\mu_i(t)\|_{L^\infty(\mathbb T^N)}\nonumber\\
\leq
&C_D
\left[
\bigl(Z_D(0)+|\Phi_D(0)|\bigr)^{1/2}e^{-\sigma_D t}
+
\bigl(K_D(T)+|\Phi_D(T)|\bigr)^{1/2}e^{-\sigma_D(T-t)}
\right].
\end{align}
\end{lemma}

\begin{proof}
The Fokker-Planck equation for $\mu_i$ is a uniformly parabolic equation in divergence form with bounded Hölder coefficients and bounded drift. Local parabolic $L^2$-to-$L^\infty$ estimates on cylinders of unit length give
$$
\|\mu_i(t)\|_{L^\infty}
\leq
C
\left(
\|\mu_i\|_{L^2((t-1,t)\times\mathbb T^N)}
+
\|\nabla v_i\|_{L^2((t-1,t)\times\mathbb T^N)}
\right)
$$
for $t\in[1,T]$, with the obvious modification near $t=0$. Since $\bar m_i$ is bounded above and below on $\mathbb T^N$, the right-hand side is controlled by the weighted quantities $Z_D$ and $K_D$. Proposition~\ref{prop:L2-turnpike} then yields the asserted estimate. The endpoint interval $[T-1,T]$ is treated in the same way, using the terminal contribution $K_D(T)+|\Phi_D(T)|$.
\end{proof}
To establish the turnpike property of $\Phi_D$, it is necessary to discuss the boundedness of $|\Phi_D(0)|$ and $|\Phi_D(T)|$, which is
{{\begin{lemma}\label{lemma36}
Under the assumptions of Proposition~\ref{prop:Phi-exponential}, there exists a constant
$C_\Phi>0$ such that
\begin{equation}\label{phi-boundary}
|\Phi_D(0)|+|\Phi_D(T)|
\le d_{\max}
\sum_{i=1}^2
\left(
\left\|
\frac{\mu_{i0}}{\sqrt{\bar m_i}}
\right\|_{L^2}
+
\left\|
\frac{\mu_{i0}}{\sqrt{\bar m_i}}
\right\|_{L^2}^{\,2}
\right)
+ d_{\max}
\sum_{i=1}^2
C_{P,i}^{\,2}
\|\sqrt{\bar m_i}\nabla v_{iT}\|_{L^2}^{\,2},
\end{equation}
where $d_{\max}=\max\{d_1,d_2\}>0$ depending on $d_1,d_2>0$ and $C_{P,i}$ denotes the weighted Poincaré constant associated with
$\bar m_i$.
\end{lemma}
\begin{proof}
First of all, for $\Theta_0$ sufficiently small given in (\ref{mu_nabla_v_small}), we have
\[
\|v_i(t)\|_{L^\infty(\mathbb T^N)}
\le \Theta_0\le 1,
\qquad i=1,2.
\]

Now, we  estimate $\Phi_D(0)$.
Using Hölder's inequality and the definition of $\Phi_D$, one has
\[
\Phi_D(0)
=
\sum_{i=1}^2 d_i
\int_{\mathbb T^N}
\left(
\mu_{i0} v_i(0)
+
\frac{(\mu_{i0})^2}{\bar m_i}
\right)dx.
\]
By using H\"{o}lder's inequality, we find
\[
\begin{aligned}
\Phi_D(0)
&\le
\sum_{i=1}^2
\left(
\left(d_i
\int_{\mathbb T^N}
\frac{(\mu_{i0})^2}{\bar m_i}\,dx
\right)^{1/2}
\left(
\int_{\mathbb T^N} d_i
\bar m_i v_i^2(0)\,dx
\right)^{1/2}
+
d_i\int_{\mathbb T^N}
\frac{(\mu_{i0})^2}{\bar m_i}\,dx
\right)
\\
&\le
\sum_{i=1}^2d_i
\left(
\left\|
\frac{\mu_{i0}}{\sqrt{\bar m_i}}
\right\|_{L^2}
+
\left\|
\frac{\mu_{i0}}{\sqrt{\bar m_i}}
\right\|_{L^2}^{\,2}
\right),
\end{aligned}
\]
where we used $\|v_i\|_{L^\infty}\le1$.  Next, we estimate $\Phi_D(T)$. We define
\[
\widehat v_i
 =
v_i-\int_{\mathbb T^N}v_i\,dx .
\]\
Note that
\[
\begin{aligned}
\Phi_D(T)
&=
\sum_{i=1}^2 d_i
\int_{\mathbb T^N}
\left(
\mu_i(T)\widehat v_i(T)
+
\frac{\mu_i^2(T)}{\bar m_i}
\right)dx .
\end{aligned}
\]
Then by Young's inequality, we have $\displaystyle \mu_i\widehat v_i\ge -\frac{\mu_i^2}{2\bar m_i} -\frac{\bar m_i\widehat v_i^2}{2}$, then it follows that
\[
\Phi_D(T)
\ge
\sum_{i=1}^2
\int_{\mathbb T^N} d_i
\left(
\frac{\mu_i^2}{2\bar m_i}
-\frac{\bar m_i\widehat v_i^2}{2}
\right)dx.
\]
By using the weighted Poincar\'{e} inequality, we have
\[
\int_{\mathbb T^N}
\bar m_i \widehat v_i^2\,dx
\le
C_{P,i}^{\,2}
\int_{\mathbb T^N}
\bar m_i |\nabla v_i|^2\,dx.
\]
Thus,
\[
\Phi_D(T)
\ge
-\frac12
\sum_{i=1}^2
C_{P,i}^{\,2}d_i
\int_{\mathbb T^N}
\bar m_i |\nabla v_i(T)|^2\,dx .
\]

For the reverse direction, with the aid of 
\eqref{turnpike_decay_phi}, we obtain
\begin{align*}
\Phi_D(0)
\ge
\Phi_D(T)e^{-\sigma_D T}
\ge&
-\frac{d_{\max}}2 e^{-\sigma_D T}
\sum_{i=1}^2
C_{P,i}^{2}
\int_{\mathbb T^N}
\bar m_i |\nabla v_i(T)|^2\,dx\\
\ge&
-\frac{d_{\max}}2
\sum_{i=1}^2
C_{P,i}^{2}
\int_{\mathbb T^N}
\bar m_i |\nabla v_i(T)|^2\,dx .
\end{align*}
Moreover, we have
\[
\Phi_D(T)
\le
\Phi_D(0)e^{-\sigma_D T}
\le d_{\max}
\sum_{i=1}^2
\left[
\left(
\int_{\mathbb T^N}
\frac{(\mu_{i0})^2}{\bar m_i}\,dx
\right)^{1/2}
+
\int_{\mathbb T^N}
\frac{(\mu_{i0})^2}{\bar m_i}\,dx
\right],
\]
which completes the proof of this lemma.
\end{proof}
}}
Now, we derive the $L^\infty$ estimates of $v_i$ and $\nabla v_i$ for $i=1,2$ in the following two lemmas.
{{\begin{lemma}\label{lemLinftyv}
Let $(v_i,\mu_i)_{i=1,2}$ be a classical solution of
\eqref{pertsystem}. Assume that there exists a constant
$C_f>0$ such that
\begin{align}\label{Lipassumption}
\left|
f_i(\bar{\mathbf m}+\boldsymbol\mu)
-
f_i(\bar{\mathbf m})
\right|
\le
C_f
\sum_{j=1}^2 |\mu_j|,
\qquad i=1,2,
\end{align}
Then, for every $t\in[0,T]$ and $i=1,2$,
\[
\|v_i(t,\cdot)\|_{L^\infty(\mathbb T^N)}
\le
\|v_{iT}\|_{L^\infty(\mathbb T^N)}
+
\frac{4}{\sigma}C_f\epsilon.
\]
\end{lemma}

\begin{proof}
We shall apply the adjoint method. To begin with,
let $\Gamma_i$ be the solution of the following adjoint equation
\begin{align}\label{Gammai_s}
\begin{cases}
\partial_s\Gamma_i
-\Delta\Gamma_i
-\nabla\cdot(\Gamma_i\nabla\bar u_i)
=0,
& (s,y)\in(t,T)\times\mathbb T^N,
\\[1ex]
\Gamma_i(t,\cdot)=\delta_x.
\end{cases}
\end{align}

Multiplying the equation for $v_i$ in (\ref{weightpertsystem}) by $\Gamma_i$ and the equation for
$\Gamma_i$ in (\ref{Gammai_s}) by $v_i$, integrating over $(t,T)\times\mathbb T^N$, and
performing integration by parts, we obtain
\[
v_i(t,x)
=
\int_{\mathbb T^N}
v_{iT}(y)\Gamma_i(T,y)\,dy
+
\int_t^T\int_{\mathbb T^N}
\left(
f_i(\bar{\mathbf m}+\boldsymbol\mu)
-
f_i(\bar{\mathbf m})
-
\frac12|\nabla v_i|^2
\right)
\Gamma_i\,dy\,ds .
\]
Since $\Gamma_i$ is a solution to (\ref{Gammai_s})  and
$-\frac12|\nabla v_i|^2\le0$, it follows that
\[
v_i(t,x)
\le
\|v_{iT}\|_{L^\infty(\mathbb T^N)}
+
\int_t^T
\left\|
f_i(\bar{\mathbf m}+\boldsymbol\mu)
-
f_i(\bar{\mathbf m})
\right\|_{L^\infty(\mathbb T^N)}
\,ds .
\]
Using the Lipschitz continuity of $f_i$ assumed in \eqref{Lipassumption}, we deduce
\[
v_i(t,x)
\le
\|v_{iT}\|_{L^\infty(\mathbb T^N)}
+
C_f
\int_t^T
\sum_{j=1}^2
\|\mu_j(s,\cdot)\|_{L^\infty(\mathbb T^N)}
\,ds .
\]
Hence, it follows from (\ref{mui_Linf_estimate}) and Lemma \ref{lemma36}  that
\[
v_i(t,x)
\le
\|v_{iT}\|_{L^\infty(\mathbb T^N)}
+
2C_f\epsilon
\int_t^T
\Bigl(
e^{-\sigma s}
+
e^{-\sigma(T-s)}
\Bigr)
\,ds ,
\]
which yields
\[
v_i(t,x)
\le
\|v_{iT}\|_{L^\infty( \mathbb{T}^N)}
+
\frac{4}{\sigma}C_f\epsilon,
\]
where the constant $C>0$ depends on $\Vert \mu_{i0}/\bar m_i\Vert_{L^2}$ and $\Vert \sqrt{\bar m_i}\nabla v_{iT}\Vert_{L^2},$ $i=1,2.$

To obtain the opposite inequality, consider the probability flow
$\bar\Gamma_i$ solving
\[
\begin{cases}
\partial_s\bar\Gamma_i
-\Delta\bar\Gamma_i
-\nabla\cdot\!\Bigl(
\bar\Gamma_i(\nabla\bar u_i+\nabla v_i)
\Bigr)
=0,
& (s,y)\in(t,T)\times\mathbb T^N,
\\[1ex]
\bar\Gamma_i(t,\cdot)=\delta_x.
\end{cases}
\]
and repeating the previous argument, we obtain
\[
v_i(t,x)
=
\int_{\mathbb T^N}
v_{iT}(y)\bar\Gamma_i(T,y)\,dy
+
\int_t^T\int_{\mathbb T^N}
\left(
f_i(\bar{\mathbf m}+\boldsymbol\mu)
-
f_i(\bar{\mathbf m})
+
\frac12|\nabla v_i|^2
\right)
\bar\Gamma_i\,dy\,ds .
\]
Since $\frac12|\nabla v_i|^2\ge0$, we find
\[
v_i(t,x)
\ge
-\|v_{iT}\|_{L^\infty\mathbb T^N)}
-
\int_t^T
\left\|
f_i(\bar{\mathbf m}+\boldsymbol\mu)
-
f_i(\bar{\mathbf m})
\right\|_{L^\infty(\mathbb T^N)}
\,ds .
\]
Consequently, one finds from (\ref{Lipassumption}), Lemma \ref{lembootstrap} and Lemma \ref{lemma36} that
\[
-v_i(t,x)
\le
\|v_{iT}\|_{L^\infty(\mathbb T^N)}
+
\frac{4}{\sigma}C_f\epsilon.
\]
Combining the upper and lower bounds gives
\[
\|v_i(t,\cdot)\|_{L^\infty(\mathbb T^N)}
\le
\|v_{iT}\|_{L^\infty(\mathbb T^N)}
+
\frac{4}{\sigma}C_f \epsilon ,
\]
which concludes the proof.
\end{proof}}}

Now, we discuss the gradient estimate of $v_i$, $i=1,2$, which is

{{\begin{lemma}\label{lem:gradtwopop}
Suppose that $(  v_i, \mu_i)_{i=1,2}$ is a classical solution of
\eqref{pertsystem}. Assume that the hypotheses of
Lemma~\ref{lemLinftyv} hold. In particular, assume that
\[
\|v_i\|_{L^\infty(  \mathbb T^N)}\le C_v,
\qquad i=1,2,
\]
for some constant $C_v\le 1$ independent of $T$.
Then there exists a constant $C'>0$, depending only on
$\bar u_1,\bar u_2$ and $C_v$, such that for every $i=1,2$,
\[
\|\nabla v_i(t,\cdot)\|_{L^\infty( \mathbb{T}^N)}
\le
4e^{C_v/2}
\|\nabla v_{iT}\|_{L^\infty( \mathbb{T}^N)}
+
2e^{C_v/2}C'
\bigl( \epsilon+C_v\bigr).
\]
\end{lemma}

\begin{proof}
 Define $w_i =e^{-v_i/2}.$Then, we have 
\[
v_i=-2\log w_i, \qquad \nabla v_i=-\frac{2}{w_i}\nabla w_i.
\]
Similarly,we obtain 
\[
\partial_t v_i=-\frac{2}{w_i}\partial_t w_i, \qquad \nabla v_i=-\frac{2}{w_i}\nabla w_i,
\]
and
\[
\Delta v_i
=
-\frac{2}{w_i}\Delta w_i
+
\frac{2}{w_i^2}|\nabla w_i|^2.
\]
Substituting these identities into the equation of $v_i$ in (\ref{pertsystem}), we obtain
\begin{align}
\left\{\begin{array}{ll}
-\partial_t w_i
-\Delta w_i
+
\nabla\bar u_i\cdot\nabla w_i
+
\frac{w_i}{2}
\Bigl(
f_i(\bar{\mathbf m}+\boldsymbol\mu)
-
f_i(\bar{\mathbf m})
\Bigr)
=0, &x\in\mathbb T^N,t\in(0,T),\\
w_i(T,\cdot)=e^{-v_{iT}/2}.
\end{array}
\right.
\end{align}

Since $\|v_i\|_{L^\infty([0,T]\times\mathbb T^N)}\le C_v,$
, we have
\[
e^{-C_v/2}
\le
w_i(t,x)
\le
e^{C_v/2},
\qquad
x\in\mathbb T^N, t\in[0,T].
\]
Moreover, letting $\widetilde w_i =w_i-1$, since $C_v\le1$, we have
\[ 
\|\widetilde w_i\|_{L^\infty([0,T]\times\mathbb T^N)}
=
\|w_i-1\|_{L^\infty([0,T]\times\mathbb T^N)}
\le C_v.
\]
Applying the linear gradient estimate to $\widetilde w_i$, one gets
\begin{align}
\sup_{t\in[0,T]} &
\|\nabla\widetilde w_i(t,\cdot)\|_{L^\infty(\mathbb T^N)}\nonumber
\\
\le&2\|\nabla\widetilde w_i(T,\cdot)\|_{L^\infty(\mathbb T^N)}
+
C
\left\|
\frac{w_i}{2}
\Bigl(
f_i(\bar{\mathbf m}+\boldsymbol\mu)
-
f_i(\bar{\mathbf m})
\Bigr)
\right\|_{L^\infty([0,T]\times\mathbb T^N)}
 \nonumber\\ & \qquad +
C\|\widetilde w_i\|_{L^\infty([0,T]\times\mathbb T^N)} .
\end{align}
Using the Lipschitz assumption, we have 
\[
\left|
f_i(\bar{\mathbf m}+\boldsymbol\mu)
-
f_i(\bar{\mathbf m})
\right|
\le
C_f\sum_{j=1}^2|\mu_j|
\]
and the estimate $
\|\mu_j\|_{L^\infty([0,T]\times\mathbb T^N)}\le C\varepsilon,~i=1,2$ shown in Lemma \ref{lembootstrap} and Lemma \ref{lemma36},
we infer that
\[
\sup_{t\in[0,T]}
\|\nabla\widetilde w_i(t,\cdot)\|_{L^\infty(\mathbb T^N)}
\le
2\|\nabla\widetilde w_i(T,\cdot)\|_{L^\infty(\mathbb T^N)}
+
C'(\varepsilon+C_v).
\]
Moreover, since $\nabla\widetilde w_i=\nabla w_i$, we have
\[
\|\nabla v_i(t,\cdot)\|_{L^\infty(\mathbb T^N)}
\le
2e^{C_v/2}
\|\nabla w_i(t,\cdot)\|_{L^\infty(\mathbb T^N)}.
\]
Furthermore, we have 
\[
\nabla \widetilde w_i(T,\cdot)
=
-\frac12 e^{-v_{iT}/2}\nabla v_{iT},
\]
and therefore we obtain 
\[
\|\nabla\widetilde w_i(T,\cdot)\|_{L^\infty(\mathbb T^N)}
\le
\frac12 e^{C_v/2}
\|\nabla v_{iT}\|_{L^\infty(\mathbb T^N)}.
\]
Consequently, we get 
\[
\sup_{t\in[0,T]}
\|\nabla v_i(t,\cdot)\|_{L^\infty(\mathbb T^N)}
\le
4e^{C_v/2}
\|\nabla\widetilde w_i(T,\cdot)\|_{L^\infty(\mathbb T^N)}
+
2e^{C_v/2}C'(\varepsilon+C_v).
\]
By the previously established bound on 
\(\|\nabla\widetilde w_i(T,\cdot)\|_{L^\infty}\), and enlarging the constant \(C'\) if necessary, we deduce that
\[
\sup_{t\in[0,T]}
\|\nabla v_i(t,\cdot)\|_{L^\infty(\mathbb T^N)}
\le
4e^{C_v/2}
\|\nabla v_{iT}\|_{L^\infty(\mathbb T^N)}
+
2e^{C_v/2}C'
\bigl(\varepsilon+C_v\bigr).
\]
This estimate completes the proof.
\end{proof}}
We are ready to give the proof of the main theorem \ref{thm:twopopturnpike}, which is

\begin{proof}
The desired a priori estimates are obtained by combining
Propositions~\ref{prop:Phi-exponential} and \ref{prop:L2-turnpike}
with Lemma~\ref{lembootstrap}.  It remains to establish the existence of a solution satisfying the prescribed smallness condition.

Let $T>0$ be fixed, and we set $\displaystyle
X =C([0,T]\times\mathbb T^N)^2$. For $\boldsymbol\rho=(\rho_1,\rho_2)\in X$, we define
\[
\|\!\|\boldsymbol\rho\|\!\|_X
 =
\sum_{i=1}^2
\sup_{t\in[0,T]}
\frac{\|\rho_i(t,\cdot)\|_{L^\infty(\mathbb T^N)}}
{e^{-\sigma_D  t}+e^{-\sigma_D (T-t)}} ,
\]
where $\sigma_D>0$ depends on $d_1,d_2>0.$
Since $T$ is fixed, this norm is equivalent to the usual norm
\[
\|\boldsymbol\rho\|_X
 =
\sum_{i=1}^2
\sup_{t\in[0,T]}
\|\rho_i(t,\cdot)\|_{L^\infty(\mathbb T^N)} .
\]
Let
\begin{align}\label{convexball_U}
U =\overline B_{\bar\varepsilon}(0)
=
\left\{
\boldsymbol\rho\in X:
\|\!\|\boldsymbol\rho\|\!\|_X\le \bar\varepsilon
\right\},
\end{align}
where $\bar\varepsilon>0$ will be chosen sufficiently small.

We now complete the Schauder-type fixed point argument.  To begin with, define $\mathcal F:U\times[0,1]\to X $
by setting $\mathcal F(\boldsymbol{\rho},{\nu})=\boldsymbol{\mu}$, where \(\boldsymbol{\mu}\) denotes the solution of the following system:
\begin{align}
\left\{
\begin{array}{ll}
-\partial_t v_i
-\frac1{\bar m_i}\nabla\cdot(\bar m_i\nabla v_i)
=
\nu
\left[
f_i(\bar{\mathbf m}+\boldsymbol\rho)
-
f_i(\bar{\mathbf m})
-\frac12|\nabla v_i|^2
\right],&x\in\mathbb T^N,t\in(0,T),
\\
\partial_t\mu_i
-\nabla\cdot(\bar m_i\nabla z_i)
-\nabla\cdot(\bar m_i\nabla v_i)
=
\nu\nabla\cdot(\mu_i\nabla v_i),&x\in\mathbb T^N,t\in(0,T),
\\
\mu_i(0,\cdot)
=
\nu(m_{i0}-\bar m_i),&x\in\mathbb T^N, 
\\
v_i(T,\cdot)
 =
\nu(u_{iT}-\bar u_i),&x\in\mathbb T^N, ~i=1,2,
\end{array}
\right.
\end{align}
with $\displaystyle z_i =\frac{\mu_i}{\bar m_i}$. For $\nu=0$, the above system reduces to a linear forward-backward parabolic system. For all $\nu\in[0,1]$, the coercivity condition remains uniform, since
$$
\nu
\left(
\int_{\mathbb T^N}\boldsymbol\mu^TDA_s\boldsymbol\mu\,dx+G_D
\right)
+
(1-\nu)G_D
\geq
\eta G_D.
$$
The fixed points of $\mathcal F(\cdot,1)$ are precisely the solutions of the
perturbation system (\ref{pertsystem}).  We next prove that $\mathcal F$ satisfies the assumptions of Leray-Schauder's
fixed point theorem.

\medskip

\noindent
\textbf{Step 1: Continuity of $\mathcal F$.}\\
Let $(\boldsymbol\rho^k,\nu_k)\to(\boldsymbol\rho,\nu)
\quad\text{in } X\times[0,1],$
and set $\boldsymbol\mu^k=\mathcal F(\boldsymbol\rho^k,\nu_k),
\boldsymbol\mu=\mathcal F(\boldsymbol\rho,\nu).$
We shall show that
\[
\boldsymbol\mu^k\to\boldsymbol\mu
\quad\text{in } X.
\]
To this end, we let $\hat v_i^k =v_i^k-v_i$. Subtracting the two equations for $v_i^k$ and
$v_i$, one obtains
\begin{align}
\begin{array}{ll}
-\partial_t\hat v_i^k
-\frac1{\bar m_i}\nabla\cdot(\bar m_i\nabla\hat v_i^k)
&=
\nu_k
\left[
f_i(\bar{\mathbf m}+\boldsymbol\rho^k)
-
f_i(\bar{\mathbf m})
\right]
\\
&\quad
-
\nu
\left[
f_i(\bar{\mathbf m}+\boldsymbol\rho)
-
f_i(\bar{\mathbf m})
\right]
\\
&\quad
-\frac12
\left(
\nu_k|\nabla v_i^k|^2
-
\nu|\nabla v_i|^2
\right),~~x\in\mathbb T^N,t\in(0,T)
\end{array}
\end{align}
with a terminal condition
\[
\hat v_i^k(T,\cdot)
=
(\nu_k-\nu)(u_{iT}-\bar u_i).
\]
Since $\boldsymbol\rho^k,\boldsymbol\rho\in U$, the estimates of
Lemma~\ref{lemLinftyv} and Lemma~\ref{lem:gradtwopop} imply that
\[
\|\nabla v_i^k\|_{L^\infty}
+
\|\nabla v_i\|_{L^\infty}
\le \Theta
\]
uniformly in $k$ for constant $\Theta>0$.   Then, arguing as in Lemma~\ref{lemLinftyv}, we obtain
\[
\begin{aligned}
\|\hat v_i^k(t,\cdot)\|_{L^\infty(\mathbb T^N)}
&\le
|\nu_k-\nu|
\|v_{iT}\|_{L^\infty(\mathbb T^N)}
\\
&\quad
+
\int_t^T
\left\|
\nu_k
\Bigl[
f_i(\bar{\mathbf m}+\boldsymbol\rho^k)
-
f_i(\bar{\mathbf m})
\Bigr]
-
\nu
\Bigl[
f_i(\bar{\mathbf m}+\boldsymbol\rho)
-
f_i(\bar{\mathbf m})
\Bigr]
\right\|_{L^\infty(\mathbb T^N)}
\,ds .
\end{aligned}
\]
By the Lipschitz continuity of $f_i$, we have
\[
\begin{aligned}
&\left\|
\nu_k
\Bigl[
f_i(\bar{\mathbf m}+\boldsymbol\rho^k)
-
f_i(\bar{\mathbf m})
\Bigr]
-
\nu
\Bigl[
f_i(\bar{\mathbf m}+\boldsymbol\rho)
-
f_i(\bar{\mathbf m})
\Bigr]
\right\|_{L^\infty}
\\
&\qquad
\le
|\nu_k-\nu|
\left\|
f_i(\bar{\mathbf m}+\boldsymbol\rho^k)
-
f_i(\bar{\mathbf m})
\right\|_{L^\infty}
+
\nu
\left\|
f_i(\bar{\mathbf m}+\boldsymbol\rho^k)
-
f_i(\bar{\mathbf m}+\boldsymbol\rho)
\right\|_{L^\infty}
\\
&\qquad
\le
C_f|\nu_k-\nu|
\sum_{j=1}^2
\|\rho_j^k\|_{L^\infty}
+
C_f
\sum_{j=1}^2
\|\rho_j^k-\rho_j\|_{L^\infty}.
\end{aligned}
\]
Therefore, we have 
\[
\begin{aligned}
\|\hat v_i^k(t,\cdot)\|_{L^\infty(\mathbb T^N)}
&\le
|\nu_k-\nu|
\|v_{iT}\|_{L^\infty(\mathbb T^N)}
\\
&\quad
+
TC_f|\nu_k-\nu|
\sum_{j=1}^2
\|\rho_j^k\|_X
+
TC_f
\sum_{j=1}^2
\|\rho_j^k-\rho_j\|_X .
\end{aligned}
\]
Then, we have 
\[
\|\hat v_i^k\|_X
\le
|\nu_k-\nu|
\left(
\|v_{iT}\|_{L^\infty}
+
TC_f\|\boldsymbol\rho^k\|_X
\right)
+
TC_f
\|\boldsymbol\rho^k-\boldsymbol\rho\|_X .
\]
Since $\boldsymbol\rho^k\to\boldsymbol\rho$ in $X$ and $\nu_k\to\nu$, the
right-hand side goes to zero. Hence
\[
\|v_i^k-v_i\|_X\to0,
\qquad i=1,2.
\]

Similarly, we have from Lemma 5.4 in \cite{cirant2025local} with $
H=-\nabla\bar u_i-\frac12 (\nabla v_i^k+\nabla v_i)$ and $g=\nu_k
\left[
f_i(\bar{\mathbf m}+\boldsymbol\rho^k)
-
f_i(\bar{\mathbf m})
\right]-
\nu
[
f_i(\bar{\mathbf m}+\boldsymbol\rho)-f_i(\bar{\mathbf m})]$
that 
\[
\begin{aligned}
\|D\hat v_i^k\|_X
&\le
2
|\,\nu_k-\nu\,|
\|Dv_{iT}\|_{L^\infty(\mathbb T^N)}
\\
&\quad
+
C C_f
\Bigl(
|\nu_k-\nu|
\|\boldsymbol\rho^k\|_X
+
\|\boldsymbol\rho^k-\boldsymbol\rho\|_X
\Bigr).
\end{aligned}
\]
Therefore, we have 
\[
\|v_i^k-v_i\|_X
+
\|\nabla v_i^k-\nabla v_i\|_X
\to 0
\qquad\text{as } k\to\infty .
\]
Now we set $\hat\mu_i^k =\mu_i^k-\mu_i$. By subtracting the equations for $\mu_i^k$ and $\mu_i$, we obtain
\begin{align}
\begin{array}{ll}
\partial_t\hat\mu_i^k
-\nabla\cdot(\bar m_i\nabla\hat z_i^k)
-\nabla\cdot(\bar m_i\nabla\hat v_i^k)
&=
\nu_k\nabla\cdot(\hat\mu_i^k\nabla v_i^k)
\\
&\quad
+
\nu_k\nabla\cdot(\mu_i\nabla\hat v_i^k)
\\
&\quad
+
(\nu_k-\nu)\nabla\cdot(\mu_i\nabla v_i),~x\in\mathbb T^N,t\in(0,T),
\end{array}
\end{align}
with
\[
\hat\mu_i^k(0,\cdot)
=
(\nu_k-\nu)(m_{i0}-\bar m_i),
\qquad
\hat z_i^k =\frac{\hat\mu_i^k}{\bar m_i}.
\]
Using the standard parabolic $L^\infty$ estimate for equations in divergence
form, together with the uniform bounds on $\mu_i^k$, $\mu_i$, $\nabla v_i^k$
and $\nabla v_i$, we infer that
\[
\|\hat\mu_i^k\|_X
\le
C|\nu_k-\nu|
+
C\|\nabla\hat v_i^k\|_X ,
\]
where $C>0$ is a constant. Consequently, we obtain
\[
\|\mu_i^k-\mu_i\|_X\to0,
\qquad i=1,2,
\]
 
which implies $\mathcal F$ is continuous.

\medskip

\noindent
\textbf{Step 2: Compactness of $\mathcal F$.}\\
Let $\boldsymbol\mu^k=\mathcal F(\boldsymbol\rho^k,\nu_k)$, $(\boldsymbol\rho^k,\nu_k)\in U\times[0,1]$. The estimates obtained above imply that
\[
\|\boldsymbol\mu^k\|_X
+
\|\mathbf v^k\|_X
+
\|\nabla\mathbf v^k\|_X
\le C
\]
uniformly in $k$. Moreover, since $\boldsymbol\rho^k$ is bounded in $X$ and
the data are smooth, standard parabolic regularity theory yields, for some
$\theta\in(0,1)$ independent of $k$,
\[
\|\boldsymbol\mu^k\|_{C^{\theta,\theta/2}([0,T]\times\mathbb T^N)}
\le C .
\]
Hence $\{\boldsymbol\mu^k\}$ is uniformly bounded and equicontinuous in $X$.
By the Arzelà–Ascoli theorem, there exists a subsequence, still denoted by
$\boldsymbol\mu^k$, and some $\boldsymbol\mu\in X$ such that
\[
\boldsymbol\mu^k\to\boldsymbol\mu
\quad\text{in }X.
\]
Therefore $\mathcal F$ is compact.

\medskip

\noindent
\textbf{Step 3: Fixed points cannot lie on $\partial U$.}\\
Let $\boldsymbol\mu=\mathcal F(\boldsymbol\mu,\nu)$ be defined for some $\nu\in[0,1]$.
Then $(v_i,\mu_i)_{i=1}^2$ solves the homotopy system. Since all terms
are in divergence form in the equation for $\mu_i$, we have
\[
\frac{d}{dt}\int_{\mathbb T^N}\mu_i(t,x)\,dx=0.
\]
Thus, we get 
\[
\int_{\mathbb T^N}\mu_i(t,x)\,dx
=
\nu\int_{\mathbb T^N}(m_{i0}-\bar m_i)\,dx
=
0.
\]
The a priori estimates from Lemma \ref{lembootstrap},
Lemma~\ref{lemLinftyv}, and Lemma~\ref{lem:gradtwopop} apply
uniformly in $\nu\in[0,1]$. Therefore, choosing the sizes of the initial and
terminal perturbations sufficiently small, we obtain
$\|\!\|\boldsymbol\mu\|\!\|_X
\le
\frac{\bar\varepsilon}{2},$ where the constant $\bar\varepsilon>0$ is defined in (\ref{convexball_U}).  It follows that  $\boldsymbol\mu\notin\partial U .$

We may now apply Leray-Schauder's fixed point theorem. Since for $\nu=0$
the system is linear and the only fixed point is the trivial one, and since
no fixed point crosses the boundary of $U$, there exists a fixed point of
$\mathcal F(\cdot,1)$ in $U$. This fixed point gives a classical solution to the perturbation system. Consequently, we obtain: 
\[
u_i=\bar u_i+v_i,
\qquad
m_i=\bar m_i+\mu_i,
\qquad i=1,2,
\]
is a classical solution of the original time-dependent system.
Finally, the estimates stated in the theorem follow from
Lemma~\ref{lembootstrap}, Lemma~\ref{lemLinftyv}, and
Lemma~\ref{lem:gradtwopop}. This finishes the proof.
 \end{proof}
 
\section{Spectral Assumption}
\label{secspectralc}
In this section, we present a spectral description of assumption \eqref{matrixS2}. Specifically, we demonstrate a strong relationship between this coercivity requirement and the eigenvalue problem  of  the  linearized  elliptic system.
We define
\begin{align}\label{mathcalH0_admissible set}
\mathcal H_0
 =
\left\{
\mathbf z=(z_1,z_2):
z_i\in H^1( \bar m_i\,dx)\cap L^\infty(\mathbb T^N),
\quad
\int_{\mathbb T^N}z_i\bar m_i\,dx=0
\right\}.
\end{align}
For $\mathbf z\in\mathcal H_0$, we denote $\mu_i =\bar m_i z_i$ and 
$$
\mathcal G_D[\mathbf z]
 =
\sum_{i=1}^2 d_i
\int_{\mathbb T^N}\bar m_i|\nabla z_i|^2\,dx,~\mathcal Z_D[\mathbf z]
 =
\sum_{i=1}^2 d_i
\int_{\mathbb T^N}z_i^2\bar m_i\,dx,
$$
$$
\mathcal B_D[\mathbf z]
 =
\mathcal G_D[\mathbf z]
+
\int_{\mathbb T^N}
\boldsymbol\mu^TDA_s(x)\boldsymbol\mu\,dx.
$$
By using the notation shown above, as shown in Lemma \ref{lemequi}, we find the condition \eqref{matrixS2} is equivalent to
\begin{equation}\label{lem:B-S2}
\mathcal B_D[\mathbf z]\geq \eta\mathcal G_D[\mathbf z],
\qquad
\forall \mathbf z\in\mathcal H_0.
\end{equation}

\begin{lemma}\label{lemselfadjoint}
If $DA=A^TD$, then $\mathcal L_{\bar{\mathbf m},D}$ is self-adjoint on
$ L^2(\bar{\mathbf m}) =L^2(\bar m_1\,dx)\times L^2(\bar m_2\,dx)
$ with quadratic form $\mathcal B_D$. In particular, we have  $\left\langle
\mathcal L_{\bar{\mathbf m}}\mathbf z,\mathbf z
\right\rangle_{L^2(\bar{\mathbf m})} =\mathcal B_D[\mathbf z]$, where 
\begin{align*}
\mathcal L_{\bar{\mathbf m},D}\mathbf z:
=
-\frac1{\bar{\mathbf m}}
\nabla\cdot\bigl(D\bar{\mathbf m}\nabla \mathbf z\bigr)
+
DA_s\,\bar{\mathbf m}\mathbf z.
\end{align*}

\end{lemma}
\begin{proof}
For $\mathbf z, \mathbf w\in\mathcal H_0$, an integration by parts gives
$$
\sum_{i=1}^2d_i
\int_{\mathbb T^N}
\left(
-\frac1{\bar m_i}
\nabla\cdot(\bar m_i\nabla z_i)
\right)
w_i\bar m_i\,dx
=d_i
\sum_{i=1}^2
\int_{\mathbb T^N}
\bar m_i\nabla z_i\cdot\nabla w_i\,dx.
$$
For the term involving the coupling, we similarly obtain
$$
\sum_{i,j=1}^2
\int_{\mathbb T^N}
d_iA_{ij}\bar m_j z_j w_i\bar m_i\,dx
=
\sum_{i,j=1}^2 
\int_{\mathbb T^N}
d_iA_{ij}\bar m_i\bar m_j w_i z_j\,dx.
$$
When $DA=A^TD$, this bilinear form is symmetric in $\mathbf z$ and $\mathbf w$, and hence the corresponding operator is symmetric. Because the form is bounded from below and its resolvent is compact on the torus, the Friedrichs extension is self-adjoint. Taking $\mathbf w=\mathbf z$ in the preceding identity yields the desired relation.
\end{proof}
Assume that the relation $DA = A^{T}D$ holds. We then introduce the principal spectral quantity
\[
\hat{\lambda}_1
=
\inf_{\mathbf z\in\mathcal H_0\setminus\{0\}}
\frac{\mathcal B_D[\mathbf z]}{\mathcal Z_D[\mathbf z]}.
\]

\begin{proposition}[Rayleigh characterization]\label{prop:rayleigh}
Assume $DA=A^TD$. Then $\lambda_1$ is the first eigenvalue of
$\mathcal L_{\bar{\mathbf m}}$ on the constraint space $\mathcal H_0$. Moreover, there exists a nonzero eigenfunction
$\mathbf z^{(1)}\in\mathcal H_0$ such that
$$
\mathcal L_{\bar{\mathbf m}}\mathbf z^{(1)}=\hat\lambda_1\mathbf z^{(1)},
$$
where
 \begin{align*}
\mathcal L_{\bar{\mathbf m}}\mathbf z:
=
-\frac1{\bar{\mathbf m}}
\nabla\cdot\bigl(D\bar{\mathbf m}\nabla \mathbf z\bigr)
+
A\,\bar{\mathbf m}\mathbf z.
\end{align*}
In addition, letting $\psi_i =\sqrt{\bar m_i}\,z_i  ,~i=1,2$, we find that the eigenvalue problem associated with  $\mathcal L_{\bar{\mathbf m}} $ can be rewritten as 
\begin{align}
\left\{\begin{array}{ll}
\displaystyle -{\Delta\psi_1}+\frac{1}{2}[{ f_1(\bar w_1^2,\bar w_2^2)-\lambda_1}]\psi_1+\sum_{j=1}^2\partial_{m_1}f_j(\bar w_1^2,\bar w_2^2)\bar w^2_j\psi_j=\hat \lambda\psi_1,&x\in\mathbb T^N,\\ \\
 \displaystyle -{\Delta\psi_2}+\frac{1}{2}[{  f_2(\bar w_1^2,\bar w_2^2)-\lambda_2}]\psi_2+\sum_{j=1}^2\partial_{m_2}f_j(\bar w_1^2,\bar w_2^2)\bar w^2_j\psi_j=\hat \lambda\psi_2,&x\in\mathbb T^N,
  \end{array}
  \right.
 \end{align}
with $(\psi _1,\psi_2)\in\mathcal  H_0$, which corresponds to the linearized eigenvalue problem associated with a semilinear system.
\end{proposition}

\begin{proof}
The form $\mathcal B_D$ is bounded from below because $DA_s$ is bounded and
$$
\left|
\int_{\mathbb T^N}
\boldsymbol\mu^TDA_s\boldsymbol\mu\,dx
\right|
\leq
Cd_{\max}\mathcal Z_D[\mathbf z].
$$
The embedding $H^1\hookrightarrow L^2$ is compact, since $\bar m_i$ is smooth and bounded above and below. Therefore, the direct method in the calculus of variations gives a minimizer of the Rayleigh quotient. The Euler-Lagrange equation for the constrained minimization problem is precisely
\begin{align}\label{Schrodinger_eigenvalue}
\mathcal L_{\bar{\mathbf m}}\mathbf z=\hat \lambda\mathbf z.
\end{align}

On the other hand, {noting that $\bar m_i>0$ and $\bar u_i$, $i=1,2$ are classical, we define $\bar w_i=\sqrt{\bar m_i}$ and $\displaystyle \bar m_i=\frac{e^{-\bar u_i}}{\int e^{-\bar u_i}dx}$, so that (\ref{ss1-recalled}) can be written as 
 \begin{align}
 \left\{
 \begin{array}{ll}
 -2\Delta \bar w_1+
 \bigl(
  f_1(\bar w_1^2,\bar w_2^2)-\lambda_1
 \bigr)\bar w_1=0,&x\in \mathbb T^N,\\
 -2\Delta \bar w_2+
 \bigl(
 f_2(\bar w_1^2,\bar w_2^2)-\lambda_2
 \bigr)\bar w_2=0,&x\in \mathbb T^N,\\
 \int_{\mathbb T^N} \bar w_1^2\,dx=\int_{\mathbb T^N} \bar w_2^2\,dx=1,
 \end{array}
  \right.
  \end{align}}
Moreover, define $\psi_i=\sqrt{\bar m_i}\,z_i,~i=1,2$, then we calculate
\begin{align*}
(\mathcal L_{\bar{\mathbf m}}\mathbf z)_i=&-\frac{\Delta\psi_i}{\bar w_i}+\frac{\Delta \bar w_i}{\bar w_i^2}\psi_i+\sum_{j=1}^2\partial_{m_i}f_j(\bar w_1^2,\bar w_2^2)\bar w_j\psi_j\\
=&-\frac{\Delta\psi_i}{\bar w_i}+\frac{ f_i(\bar w_1^2,\bar w_2^2)-\lambda_i}{2\bar w_i}\psi_i+\sum_{j=1}^2\partial_{m_j}f_i(\bar w_1^2,\bar w_2^2)\bar w_i\bar w_j\psi_j,
\end{align*}
which completes the proof of this lemma.
\end{proof}
\begin{theorem}\label{thmspeccharac}
Suppose that $DA=A^TD$. Then $\hat\lambda_1>0$ is equivalent to the validity of \eqref{matrixS2}. As a result, under the assumptions of Theorem~\ref{thm:twopopturnpike}, the strict positivity of $\hat\lambda_1$ provides a spectral sufficient condition for the local exponential turnpike property.
\end{theorem}
\begin{proof}
Assume first that \eqref{matrixS2} holds. Then, by Lemma~\ref{lem:B-S2}, we have
$$
\mathcal B_D[\mathbf z]\geq \eta\mathcal G_D[\mathbf z].
$$
Using the weighted Poincaré inequalities, we obtain
$$
\mathcal Z_D[\mathbf z]\leq C_P\mathcal G_D[\mathbf z],
\qquad
C_P =\max\{C_{P,1},C_{P,2}\}.
$$
Consequently, we have 
$$
\mathcal B_D[\mathbf z]
\geq
\frac{\eta}{C_P}\mathcal Z_D[\mathbf z],
$$
which implies that $\hat \lambda_1\geq \frac{\eta}{C_P}>0$. Conversely, suppose that $\hat \lambda_1>0$. Since $DA_s$ and $\bar m_i$ are bounded,
there exists a constant $C_A>0$ such that
$$
\left|
\int_{\mathbb T^N}
\boldsymbol\mu^TDA_s\boldsymbol\mu\,dx
\right|
\leq
C_A\mathcal Z_D[\mathbf z].
$$
We define $
r_D(\mathbf z) =\frac{\mathcal G_D[\mathbf z]}{\mathcal Z_D[\mathbf z]}$. By the Poincar\'{e} inequality, we get  $ r_D(\mathbf z)\geq \frac{d_{\min}}{d_{\max}C_P}$. If $r_D(\mathbf z)\leq 2C_A+1$, then
$$
\frac{\mathcal B_D[\mathbf z]}{\mathcal G_D[\mathbf z]}
=
\frac{\mathcal B_D[\mathbf z]}{\mathcal Z_D[\mathbf z]}
\frac{\mathcal Z_D[\mathbf z]}{\mathcal G_D[\mathbf z]}
\geq
\frac{\hat \lambda_1}{2C_A+1}.
$$
On the other hand, if $r_D(\mathbf z)>2C_A+1$, then
$$
\mathcal B_D[\mathbf z]
\geq
\mathcal G_D[\mathbf z]-C_A\mathcal Z_D[\mathbf z]
=
\left(1-\frac{C_A}{r_D(\mathbf z)}\right)\mathcal G_D[\mathbf z]
\geq
\frac12\mathcal G_D[\mathbf z].
$$
In either case, we conclude that $\displaystyle \mathcal B_D[\mathbf z]\geq \eta\mathcal G_D[\mathbf z]$, where $\displaystyle
\eta = \min\left\{
\frac{\hat\lambda_1}{2C_A+1},
\frac12
\right\}>0$. By Lemma~\ref{lem:B-S2}, this is precisely \eqref{matrixS2}. This completes the proof.
\end{proof}
We now state the result for the case of non-symmetric couplings.
\begin{corollary}\label{propnonsymspectral}
If $DA\neq A^TD$, define the following Rayleigh quotient
$$
\hat \lambda_1^s = \inf_{\mathbf z\in\mathcal H_0\setminus\{0\}}
\frac{
\mathcal G_D[\mathbf z]
+
\int_{\mathbb T^N}
\boldsymbol\mu^TDA_s(x)\boldsymbol\mu\,dx
}{
\mathcal Z_D[\mathbf z]
}.
$$
Then, we have  $\hat \lambda_1^s>0$ is equivalent to  \eqref{matrixS2} holds.
\end{corollary}

\begin{proof}
The proof is identical to that of Theorem~\ref{thmspeccharac}. Since we have $
\boldsymbol\mu^TDA_{\rm sk}\boldsymbol\mu=0$. for every skew-symmetric matrix $DA_{\rm sk}=\frac{DA-A^TD}{2}$, the associated quadratic form depends only on the symmetric part $DA_s$.
\end{proof}

\section{Second-Order Variational Analysis of Local Minimizers in Two-Population Potential MFGs}\label{secpotentialenergy}

In this section, we shall discuss the relation between \eqref{matrixS2} and the second variation of the variational structure associated with the pontential two-population MFGs.  We remark that the existence of global and local minimizers for mean-field game systems with non-monotone couplings has been established in several works, see e.g. \cite{cesaroni2018concentration,cirant2025critical,kong2026blow,kong2026local}.  

Assume that the coupling is of potential type, namely, there exists a function $F\in C^3((0,\infty)^2)$ such that $f_i(m_1,m_2)=\partial_{m_i}F(m_1,m_2)$,
$i=1,2$. Then, we focus on the following constrained energy
\begin{align}\label{energyMultiple}
\mathcal E(m_1,m_2,w_1,w_2)
=&
\sum_{i=1}^2
\int_{\mathbb T^N}
\frac{|w_i|^2}{2m_i}\, dx
+
\int_{\mathbb T^N}F(m_1,m_2)\,dx
\end{align}
over  the admissible set 
\[
\mathcal A  =
\left\{
(m_i,w_i)\in (W^{1,\infty}(\mathbb{T}^N)\times L^\infty(\mathbb   T^N) )^2 | 
 -\int_{\mathbb{R}^N} \nabla m_i \cdot \nabla \varphi\,dx
+\int_{\mathbb{R}^N} w_i \cdot \nabla \varphi\,dx =0,
\right.
\]
\[
\left.
 \forall \varphi\in C_c^\infty(\mathbb{R}^N),  
\int_{\mathbb{R}^N} m_i\,dx =1, 
m_i> 0   
\right\}.
\]
 Noting that $m_i$ and $w_i$ are periodic, the energy $\mathcal E(m_1,m_2,w_1,w_2)$ given in (\ref{energyMultiple}) can be written as 
\begin{align}\label{energyMultiple1}
\mathcal E(m_1,m_2)
=&
\sum_{i=1}^2
\int_{\mathbb T^N}
\frac{|\nabla m_i|^2}{2m_i}\, dx
+
\int_{\mathbb T^N}F(m_1,m_2)\,dx\nonumber\\
=&
2\sum_{i=1}^2\int_{\mathbb T^N}
 |\nabla \sqrt{m_i}|^2 \,dx+
\int_{\mathbb T^N}F(m_1,m_2)\,dx.
\end{align}
over the admissible set
\begin{align}
\mathcal A =\{\boldsymbol{m}=(m_1,m_2)\in (W^{1,\infty}(\mathbb T^N))^2|~m_i>0,~\int_{\mathbb T^N}m_i\,dx=1,i=1,2\}.
\end{align}

\begin{lemma}[First variation]\label{lem:first-variation}
Let $\bar{\mathbf m}\in (W^{1,\infty}(\mathbb T^N))^2$ be  positive, and assume
$$
\mu_i\in H^1 (\mathbb T^N)\cap L^\infty(\mathbb T^N),
\qquad
\int_{\mathbb T^N}\mu_i\,dx=0.
$$
Then
$$
\nabla\mathcal E(\bar{\mathbf m})[D\boldsymbol\mu]
=
\sum_{i=1}^2
\int_{\mathbb T^N}
\left[
-\Delta\log\bar m_i
-\frac12|\nabla\log\bar m_i|^2
+
\partial_{m_i}F(\bar{\mathbf m})
\right]d_i\mu_i\,dx,
$$
\end{lemma}
where $D=\text{diag}\{d_1,d_2\}$.

\begin{proof}
For the Fisher-information part,
$$
\mathcal I_i(m_i)
 =
\int_{\mathbb T^N}\frac{|\nabla m_i|^2}{2m_i}\,dx.
$$
Taking $m_i^\varepsilon=\bar m_i+\varepsilon d_i\mu_i$ with $\varepsilon$ small enough, one obtains $m_i^{\varepsilon}>0$ a.e. and
$$
\frac{d}{d\varepsilon}\bigg|_{\varepsilon=0}
\mathcal I_i(m_i^\varepsilon)
=
{d_i}\int_{\mathbb T^N}
\left(
\frac{\nabla\bar m_i\cdot\nabla\mu_i}{\bar m_i}
-
\frac{|\nabla\bar m_i|^2}{2\bar m_i^2}\mu_i
\right)dx.
$$
Integrating the first term by parts gives
$$
\frac{d}{d\varepsilon}\bigg|_{\varepsilon=0}
\mathcal I_i(m_i^\varepsilon)
=
\int_{\mathbb T^N}
\left(
-\Delta\log\bar m_i
-\frac12|\nabla\log\bar m_i|^2
\right){d_i}\mu_i\,dx.
$$
Combining this with  the variations of $\int_{\mathbb T^N}F(m_1,m_2)\,dx$, we finish the proof of this lemma.
\end{proof}

\begin{lemma}\label{lem:critical-stationary}
Let $\bar{\mathbf m}$ be a constrained critical point of the functional $\mathcal E$ over $\mathcal A$. Define
$$
\bar u_i  = -\log \bar m_i + c_i,
$$
for some constants $c_i \in \mathbb R$. Then there exist constants $\lambda_i \in \mathbb R$ such that
$$
(\bar u_1, \bar u_2, \bar m_1, \bar m_2, \lambda_1, \lambda_2)
$$
constitutes a classical solution of \eqref{ss1-recalled}. Conversely, any stationary classical solution satisfying \eqref{S1} corresponds to a constrained critical point of $\mathcal E$.
\end{lemma}

\begin{proof}
Since $\bar{\mathbf m}$ is a constrained critical point, Lemma~\ref{lem:first-variation} implies the existence of constants $\lambda_i$ such that
$$
-\Delta\log\bar m_i
-\frac12|\nabla\log\bar m_i|^2
+
\partial_{m_i}F(\bar{\mathbf m})
=
\lambda_i.
$$
By the standard elliptic regularity result, we have $\bar m_i$, $i=1,2$ are classical.
Using $\bar u_i=-\log\bar m_i+c_i$, this becomes
$$
-\Delta\bar u_i+\frac12|\nabla\bar u_i|^2+\lambda_i
=
 \partial_{m_i}F(\bar{\mathbf m}).
$$
Moreover, we have $\nabla\bar m_i=-\bar m_i\nabla\bar u_{i}$; hence we get 
$$
\Delta\bar m_i+\nabla\cdot(\bar m_i\nabla\bar u_i)=0.
$$
Thus \eqref{ss1-recalled} holds. The converse follows by reversing the same computation.
\end{proof}

\begin{lemma}\label{lem:second-variation}
Let $\mu_i\in L^\infty(\mathbb T^N)$, $\int_{\mathbb T^N}\mu_i\,dx=0$. Then, we have
$$
\nabla^2\mathcal E(\bar{\mathbf m})[D\boldsymbol\mu,\boldsymbol\mu]
=
\sum_{i=1}^2
\int_{\mathbb T^N} d_i
\bar m_i
\left|
\nabla\left(\frac{\mu_i}{\bar m_i}\right)
\right|^2dx
+
\int_{\mathbb T^N}
\boldsymbol\mu^T
D\nabla_{\mathbf m}^2F(\bar{\mathbf m})
\boldsymbol\mu\,dx.
$$
\end{lemma}

\begin{proof}
Define $z_i = \frac{\mu_i}{\bar m_i},$
so that $\bar m_i + \varepsilon {d_i}\mu_i
=
\bar m_i(1 + \varepsilon {d_i}z_i).
$
A straightforward computation yields

$$
\nabla^2\mathcal I_i(\bar m_i)[d_i\mu_i,\mu_i]
=
d_i\int_{\mathbb T^N} \bar m_i|\nabla z_i|^2\,dx.
$$
Combining this with  the second variation of $\int_{\mathbb T^N}F(m_1,m_2)\,dx$, we finish the proof of this lemma.
\end{proof}

\begin{corollary}\label{corminimizerS2}
Assume that $f_i=\partial_{m_i}F$. Then, we have  $A_s(x)=D_{\mathbf m}^2F(\bar{\mathbf m}(x))$. Consequently, the nondegeneracy condition
\begin{align}\label{nondegneracy-variation}
\nabla^2\mathcal E(\bar{\mathbf m})[D\boldsymbol\mu,\boldsymbol\mu]
\geq
\eta
\sum_{i=1}^2
\int_{\mathbb T^N} d_i
\bar m_i
\left|
\nabla\left(\frac{\mu_i}{\bar m_i}\right)
\right|^2dx
\end{align}
for all mass-preserving $\boldsymbol\mu$ is exactly \eqref{matrixS2}.
\end{corollary}

\begin{proof}
Since $f_i=\partial_{m_i}F$, we have 
$$
A_{ij}
=
\partial_{m_j}f_i
=
\partial_{m_i m_j}^2F.
$$
Thus $A=A^T=D_{\mathbf m}^2F$. The conclusion follows directly from Lemma~\ref{lem:second-variation} and Lemma \ref{lemequi}.
\end{proof}

\begin{theorem}
\label{thmpotentialturnpike}
Assume $f_i=\partial_{m_i}F$. Let $\bar{\mathbf m}$ be a  positive constrained critical point of $\mathcal E$ over $\mathcal A$, and define
$$
\bar u_i=-\log\bar m_i+c_i.
$$
Assume that $\bar{\mathbf m}$ is nondegenerate in the sense of Proposition~\ref{corminimizerS2}. Then the stationary solution generated by $\bar{\mathbf m}$ satisfies the local exponential turnpike theorem of Theorem~\ref{thm:twopopturnpike}.
\end{theorem}

\begin{proof}
By Lemma~\ref{lem:critical-stationary}, $\bar{\mathbf m}$ generates a stationary MFG solution satisfying \eqref{S1}. By Proposition~\ref{corminimizerS2}, the nondegeneracy of the constrained second variation is exactly \eqref{matrixS2}. Therefore, all hypotheses of Theorem~\ref{thm:twopopturnpike} are satisfied, and the local exponential turnpike estimate follows.
\end{proof}

{{
 \begin{remark}
We remark that if a critical point of \eqref{energyMultiple1} over the admissible set $\mathcal A$ is a local or global minimizer, then, by definition, we have $D^2\mathcal E(\bar{\mathbf m})[\boldsymbol\mu,\boldsymbol\mu]\geq 0$ by definition. However, the nonnegativity of the second variation alone is not sufficient to establish \eqref{nondegneracy-variation}. To obtain \eqref{nondegneracy-variation}, one must further analyze the kernel of the linearized operator $\mathcal L_{\bar{\mathbf m}}$ defined in \eqref{Schrodinger_eigenvalue} and prove that it is trivial.  We expect that the above nondegeneracy property holds on the bounded domain $\Omega$ with Neumann boundary conditions.  In contrast, on Torus $\mathbb T^N$ or whole space $\mathbb R^N$, the translational invariance of the problem generally gives rise to nontrivial kernel elements of the linearized operator. Therefore, the kernel space is expected to be nontrivial in $\mathbb T^N$/$\mathbb R^N$.
\end{remark}

 }}

\section{Cross-interaction thresholds and non-potential interactions}
\label{seccrossnonpotential}

We now consider linear cross-interactions
\begin{align}\label{linear_crossinteraction_f}
f_1(m_1,m_2)=a_{11}m_1+a_{12}m_2,
\qquad
f_2(m_1,m_2)=a_{21}m_1+a_{22}m_2.
\end{align}
Then
$$
A=
\begin{pmatrix}
a_{11}&a_{12}
\\
a_{21}&a_{22}
\end{pmatrix},
\qquad
DA_s=
\begin{pmatrix}
d_1a_{11}&\frac{d_1a_{12}+d_2a_{21}}2
\\
\frac{d_1a_{12}+d_2a_{21}}2&d_2a_{22}
\end{pmatrix}.
$$

\begin{lemma}\label{lemweightedthreshold}
Assume that there exists a constant $\kappa \geq 0$ such that
\[
\boldsymbol{\mu}^TD A_s \boldsymbol{\mu}
 \geq 
-\kappa
\sum_{i=1}^2d_i
\frac{\mu_i^2}{\bar m_i}
\qquad
\text{for all } \boldsymbol{\mu} \in \mathbb{R}^2
\text{ and all } x \in \mathbb{T}^N.
\]
Furthermore, if the condition $\kappa C_P < 1$ is satisfied, where $C_P  = \max\{C_{P,1}, C_{P,2}\}$, then inequality \eqref{matrixS2} holds with the constant $\eta = 1 - \kappa C_P$.
\end{lemma}

\begin{proof}
By the assumed pointwise inequality,
$$
\int_{\mathbb T^N}
\boldsymbol\mu^T DA_s\boldsymbol\mu\,dx
\geq
-\kappa
\sum_{i=1}^2
\int_{\mathbb T^N} d_i
\frac{\mu_i^2}{\bar m_i}\,dx.
$$
By the weighted Poincaré inequality, we get 
$$
\sum_{i=1}^2
\int_{\mathbb T^N}
\frac{\mu_i^2}{\bar m_i}\,dx
\leq
C_P
\sum_{i=1}^2
\int_{\mathbb T^N}
\bar m_i
\left|
\nabla\left(\frac{\mu_i}{\bar m_i}\right)
\right|^2dx.
$$
Therefore, we have 
$$
\int_{\mathbb T^N}
\left[
\boldsymbol\mu^TDA_s\boldsymbol\mu
+
\sum_{i=1}^2
d_i\bar m_i
\left|
\nabla\left(\frac{\mu_i}{\bar m_i}\right)
\right|^2
\right]dx
\geq
(1-\kappa C_P)
\sum_{i=1}^2
\int_{\mathbb T^N} d_i
\bar m_i
\left| 
\nabla\left(\frac{\mu_i}{\bar m_i}\right)
\right|^2dx.
$$
This is \eqref{matrixS2}. This completes the proof. 
\end{proof}
When  $f_1$ and $f_2$ satisfy the linear cross-interactions and $DA=A^TD$, we verify the matrix coercivity condition \eqref{matrixS2} by analyzing the eigenvalue problem \eqref{Schrodinger_eigenvalue}.  The resulting operator coincides with the linearized operator associated with the coupled cubic Schrödinger system, which is
\begin{align}\label{lep_schrodinger}
\left\{\begin{array}{ll}
-{\Delta\psi_1}+\bigg[\frac{3a_{11}}{2}\bar w_1^2+\frac{a_{12}}{2}\bar w_2^2-\frac{\lambda_1}{2}\bigg]\psi_1+a_{12}\bar w_1\bar w_2\psi_2=\hat\lambda\psi_1,&x\in\mathbb T^N,\\
-{\Delta\psi_2}+\bigg[\frac{3a_{21}}{2}\bar w_1^2+\frac{a_{22}}{2}\bar w_2^2-\frac{\lambda_2}{2}\bigg]\psi_2+a_{21}\bar w_1\bar w_2\psi_1=\hat \lambda\psi_2,&x\in\mathbb T^N,
\end{array}
\right.
\end{align}
where $(\bar m_1,\bar m_2)=(\bar w_1^2,\bar w_2^2)$ is the solution to (\ref{ss1-recalled}) and $(\psi_1,\psi_2)\in \mathcal H_0$ is given by (\ref{mathcalH0_admissible set}).  { {Concerning the couplings satisfying \eqref{linear_crossinteraction_f}, it is well known that, in the strong competition regime $a_{12}=a_{21}\to -\infty$, the system \eqref{ss1-recalled} exhibits phase separation.   More precisely, one expects there to exist solutions such that  $\bar m_1\bar m_2\to 0$ as $a_{12}=a_{21}\rightarrow -\infty$.  
See, for instance, \cite{noris2010uniform,wei2008radial}.}  Our result, stated in Theorem~\ref{thm:twopopturnpike}, implies that many phase-separated states satisfy a local turnpike property.}

{ 
\subsection{Phase separation}
We focus on the phase-separation phenomenon in a two-population MFG system. To avoid the translational invariance inherent in the periodic setting $\mathbb T^N$, we consider an example posed on a bounded domain $\Omega$ with homogeneous Neumann boundary conditions.
  Let \(N=1\), \(\Omega=[-\pi,\pi]\), and suppose that
\eqref{linear_crossinteraction_f} holds with \(a_{11}=a_{22}=-2\), \(a_{12}=a_{21}=2\beta>0\). Replacing the periodic boundary conditions in \eqref{ss1-recalled} with homogeneous Neumann boundary conditions and applying the Hopf-Cole transform, we obtain the following cubic Schr\"{o}dinger system:
\begin{align}\label{eq:competitive-schrodinger-torus}
\begin{cases}
-\bar w_1''-\dfrac{\lambda_1}{2}\,\bar w_1=\bar w_1^3-\beta \bar w_1 \bar w_2^2,& x\in (-\pi,\pi),\\[0.3em]
-\bar w_2''-\dfrac{\lambda_2}{2}\,\bar w_2=\bar w_2^3-\beta \bar w_2 \bar w_1^2, & x\in (-\pi,\pi),\\[0.3em]
 \displaystyle\int_{-\pi}^\pi\bar w_1^2\,dx
=
 \int_{-\pi}^\pi\bar w_2^2\,dx
=1,\\[0.3em]
\bar w'_i(-\pi)=\bar w_i'(\pi)=0,&i=1,2,
\end{cases}
\qquad,
\end{align}
where \((\bar m_1,\bar m_2)=(\bar w_1^2,\bar w_2^2)\). As established in \cite{noris2010uniform}, when the competition is strong, Schrödinger systems of this type display phase separation in the limit \(\beta\to+\infty\). Inspired by this behavior, we focus on a symmetric segregated configuration whose limiting profiles are given by
\begin{align}\label{eq:explicit-bump}
\bar w_1^\infty(x)=
\begin{cases}
W(x+\pi),&x\in(-\pi,0),\\
0,&x\in(0,\pi),
\end{cases}
\qquad
\bar w_2^\infty(x)=
\begin{cases}
0,&x\in(-\pi,0),\\
W(\pi-x),&x\in(0,\pi),
\end{cases}
\end{align}
where \((W,\lambda)\), with \(\lambda_1=\lambda_2=2\lambda\), is the unique positive “bump’’ solution of
\[
\begin{cases}
-W''-\lambda W=W^3,
& x\in(0,\pi),
\\
W>0,\\
W'(0)=W(\pi)=0,\quad \displaystyle\int_{0}^{\pi} W^2\,dx=1.
\end{cases}
\]
Moreover, \(W\) has the explicit form
\[
W(x)
=
\sqrt{2m}\,\frac{K(m)}{\pi}\,
\mathrm{cn}\!\left(\frac{K(m)}{\pi}x;m\right),
\qquad x\in(0,\pi),
\]
where \(\mathrm{cn}(\cdot;m)\) is the Jacobi elliptic cosine with modulus
\(m\in(1/2,1)\). Here
\[
K(m)
=
\int_0^{\pi/2}
\frac{d\theta}{\sqrt{1-m\sin^2\theta}},
\qquad
E(m)
=
\int_0^{\pi/2}
\sqrt{1-m\sin^2\theta}\,d\theta
\]
denote the complete elliptic integrals of the first and second kinds, respectively. The associated Lagrange multiplier is
\[
\lambda
=
(1-2m)
\left(\frac{K(m)}{\pi}\right)^2
\]
and the parameter \(m\in(1/2,1)\) is uniquely fixed by the normalization condition
\[
\int_{0}^{\pi}W^2\,dx=1,
\]
which is equivalent to
\[
2K(m)\Big[E(m)-(1-m)K(m)\Big]=\pi.
\]

In the case \(N=1\), \(a_{11}=a_{22}=-2\) and \(a_{12}=a_{21}=2\beta>0\), we consider the linearized eigenvalue problem (\ref{lep_schrodinger}) around \((\bar m_1,\bar m_2)=(\bar w^2_1,\bar w^2_2)\) and obtain
\begin{align}\label{eigenphase}
\left\{
\begin{array}{ll}
-\psi_1''+
\left(
-3\bar w_1^2+\beta \bar w_2^2-\frac{\lambda_1}{2}
\right)\psi_1
+
2\beta \bar w_1\bar w_2\,\psi_2
=
\hat\lambda\,\psi_1,\quad &x\in(-\pi,\pi),
\\ 
-\psi_2''+
\left(
\beta \bar w_1^2-3\bar w_2^2-\frac{\lambda_2}{2}
\right)\psi_2
+
2\beta \bar w_1\bar w_2\,\psi_1
=
\hat\lambda\,\psi_2,\quad &x\in(-\pi,\pi),\\
\int_{-\pi}^\pi \bar w_i\psi_i\,dx=0,&i=1,2,\\[0.3em]
\psi_i'(-\pi)=\psi'_i(\pi)=0,&i=1,2.
\end{array}
\right.
 \end{align}
As \(\beta\rightarrow+\infty\), the formal limiting problem of (\ref{eigenphase}) becomes 
\begin{align}\label{eigenphase1}
\left\{
\begin{array}{ll}
-\psi_1''+
\left(
-3\bar w_1^2 -\frac{\lambda_1}{2}
\right)\psi_1
=
\hat\lambda\,\psi_1,\quad &x\in(-\pi,0),
\\ 
-\psi_2''+
\left(
 -3\bar w_2^2-\frac{\lambda_2}{2}
\right)\psi_2
=
\hat\lambda\,\psi_2,\quad &x\in(0,\pi),\\
\psi_1'(-\pi)=\psi_2'(\pi)=0,\\
\psi_1(0)
+
\eta(\bar w_1^\infty)'(0^-)
=0,\\
\psi_2(0)
+
\eta(\bar w_2^\infty)'(0^+)
=0,\\
(\bar w_1^\infty)'(0^-)\psi_1'(0^-)
=
(\bar w_2^\infty)'(0^+)\psi_2'(0^+),
\
\end{array}
\right.
 \end{align}
where $\eta$ denotes the perturbation of the location of the segregation interface. We next exploit the symmetry of the limiting profiles $\bar w_1^{\infty}$ and $\bar w_2^{\infty}$; we introduce a transformation that decomposes the limiting free-boundary spectral problem into the direct sum of a constrained Dirichlet-Neumann problem and a constrained Neumann-Neumann problem. To this end, we first define the symmetric and antisymmetric variables as follows
\[
p=\psi_1(-x)+\psi_2(x),
\qquad
q=\psi_1(-x)-\psi_2(x).
\]
 
Using
\[
(\bar w_1^{\infty})'(0^-)=-W'(0),
\qquad
(\bar w^{\infty}_2)'(0^+)=W'(0),
\]
we obtain
 
\begin{align}\label{problem67}
\left\{\begin{array}{ll}L_+p=\hat\lambda p,
& 0<x<\pi,\\[1mm]
p(0)=p'(\pi)=0,\\
\int_0^\pi Wp\,dx=0,
\end{array}
\right.,~\quad\quad
\left\{\begin{array}{ll}
L_+q=\hat\lambda q,
& 0<x<\pi,\\[1mm]
q'(0)=q'(\pi)=0,\\
\int_0^\pi Wq\,dx=0,
\end{array}
\right.
\end{align}
where 
\[
L_+
:=
-\frac{d^2}{dx^2}
-
3W^2
-
{\lambda},
\]
and \(\eta\) satisfies $\eta=\frac{q(0)}{2W'(0)}.$

Introducing the change of variable \(y=kx\), and using \eqref{eq:explicit-bump}, we define
 \begin{align}\label{lame_operator}
L_{+}
=
k^2\Big[
-\partial_{yy}
-6m\,\mathrm{cn}^2(y;m)
+ {2m-1} 
\Big],
\end{align}
which is a shifted Lamé operator of order \(2\). With this operator, after a translation and the substitution \(y=kx\), problem \eqref{problem67} is equivalent to
\begin{align}\label{eigenphase1equiv}
\left\{\begin{array}{ll}L_+p=\hat\lambda p,
& 0<y<k\pi,\\[1mm]
p(0)=p'(k\pi)=0,\\
\int_0^{k\pi} Wp\,dy=0,
\end{array}
\right.,~\quad\quad
\left\{\begin{array}{ll}
L_+q=\hat\lambda q,
& 0<y<k\pi,\\[1mm]
q'(0)=q'(k\pi)=0,\\
\int_0^{k\pi} Wq\,dy=0.
\end{array}
\right.
\end{align}
In the absence of the integral constraint, the eigenvalue problem \eqref{eigenphase1equiv} has the following explicit eigenpairs:
\begin{equation}\label{hatpsi1_lame}
\begin{array}{ll}
\hat\lambda^{(1)}=0,
&
p^{(1)}(y)=\operatorname{sn}(y;m)\operatorname{dn}(y;m),
\\[2mm]
\hat\lambda_{\pm}^{(2)}=(1-2m\pm2\Delta_m)k^2,
&
q_{\pm}^{(1)}(y)=1-\bigl(1+m\pm\Delta_m\bigr)\operatorname{sn}^2(y;m).
\end{array}
\end{equation}
where $\Delta_m:=\sqrt{1-m+m^2}$ and \(\mathrm{sn}(\cdot;m)\), \(\mathrm{cn}(\cdot;m)\), and \(\mathrm{dn}(\cdot;m)\) are the Jacobi elliptic sine, cosine, and delta-amplitude functions with modulus \(m\in(1/2,1)\), respectively. In particular, by the classical Sturm–Liouville theory,  
\[
\hat\lambda^{(2)}_{-}<\hat\lambda^{(1)}=0<\hat\lambda^{(2)}_{+}\le \hat\lambda^{(j)},
\qquad j\ge 3,
\]
where $\hat\lambda^{(j)}$ denotes the other eigenvalues. In light of the integral constraints, we observe that neither \(p^{(1)}\) nor \(q_{\pm}^{(1)}\) satisfies $\int_0^{k\pi} Wp\,dy=0$, $\int_0^{k\pi} Wq\,dy=0,$ as can be verified from \eqref{hatpsi1_lame}.

 
Therefore, condition \eqref{matrixS2} holds by using  Theorem \ref{thmspeccharac}.  An application of Theorem~\ref{thm:twopopturnpike} then yields the local turnpike stability of the phase-separated profile, whose limiting profile is given by \eqref{eq:explicit-bump}.}
\subsection{Sharp threshold in the homogeneous case}

Assume that on $\mathbb T^N$ we have $\bar m_1 = \bar m_2 = 1$. We denote by $\lambda_{\mathbb T} > 0$ the smallest nonzero (strictly positive) eigenvalue of the Laplacian $-\Delta$ on $\mathbb T^N$. The corresponding Poincaré constant is then $C_P = \frac{1}{\lambda_{\mathbb T}}$.

\begin{proposition}\label{prop:sharp-threshold} In the homogeneous case $\bar m_1=\bar m_2=1$, the matrix coercivity condition \eqref{matrixS2} holds if and only if $DA_s+\lambda_{\mathbb T}I_2>0$. This is furthermore equivalent to the following system of inequalities:
\[
a_{11}+\lambda_{\mathbb T}>0,
\quad
a_{22}+\lambda_{\mathbb T}>0,
\quad \text{ and}\qquad
(a_{11}+\lambda_{\mathbb T})
(a_{22}+\lambda_{\mathbb T})
>
\frac{(d_1a_{12}+d_2a_{21})^2}{4d_1d_2}.
\]
The associated sharp threshold can  be expressed as
\[
\frac{a_{11}+a_{22}}2
-
\frac12
\sqrt{
(a_{11}-a_{22})^2+\frac{(d_1a_{12}+d_2a_{21})^2}{d_1d_2}
}
>
-\lambda_{\mathbb T}.
\]
\end{proposition}

\begin{proof}
Since $\bar m_1=\bar m_2=1$, it follows that $z_i=\mu_i$. For functions $z_i$ with vanishing spatial mean, we employ their Fourier representation,
\[
\mathbf{funcZ}(x)
=
\sum_{k\neq 0}\widehat{\mathbf{funcZ}}_k\,e^{i k\cdot x},
\]
where the sum is taken over all nonzero wavevectors $k$.

In this representation, the quadratic form $\mathcal B_D$ can be expressed as
$$
\mathcal B_D[\mathbf z]
=
\sum_{k\neq0}
\widehat{\mathbf z}_k^*
\left(
|k|^2D+DA_s
\right)
\widehat{\mathbf z}_k,
$$
while the constraint functional takes the form
$$
\mathcal G_D[\mathbf z]
=
\sum_{k\neq0}|k|^2\widehat{\mathbf z}_k^*D\widehat{\mathbf z}_k.
$$
Thus, $\mathcal B_D$ is positive on the constraint space if and only if
$|k|^2I_2+DA_s>0$ for every $k\neq0$. Since the spectrum of $-\Delta$ on $\mathbb T^N$ satisfies $|k|^2\geq \lambda_{\mathbb T}$ for all $k\neq0$, this requirement is equivalent to $DA_s+\lambda_{\mathbb T}D>0$. Equivalently, this requirement can be written as $\lambda_{\min}(D^{-1/2}(DA_s)D^{1/2})>-\lambda_{\mathbb T}$. For a real symmetric $2\times2$ matrix, the positive definiteness of
$$
DA_s+\lambda_{\mathbb T}D
=
\begin{pmatrix}
d_1a_{11}+d_1\lambda_{\mathbb T}
&
\frac{d_1a_{12}+d_2a_{21}}2
\\
\frac{d_1a_{12}+d_2a_{21}}2
&
d_2a_{22}+d_2\lambda_{\mathbb T}
\end{pmatrix}
$$
is equivalent to the positivity of its two leading principal minors, that is, $ a_{11}+\lambda_{\mathbb T}>0$
and $
d_1d_2(a_{11}+\lambda_{\mathbb T})
(a_{22}+\lambda_{\mathbb T})
-
\frac{(d_1a_{12}+d_2a_{21})^2}{4}
>0$. 

The eigenvalues of $DA_s$ are given by
$$
\lambda_{\pm}(D^{-1/2}(DA_s)D^{-1/2})
=
\frac{a_{11}+a_{22}}2
\pm
\frac12
\sqrt{
(a_{11}-a_{22})^2+\frac{(d_1a_{12}+d_2a_{21})^2
}{d_1d_2}}.
$$
Consequently, we have $\lambda_{\min}(DA_s)=\lambda_-(DA_s)$, and inserting this into the criterion from Proposition~\ref{prop:sharp-threshold} yields the desired inequality. This concludes the proof of the proposition.
\end{proof}
\subsection{Non-potential interactions}

The system is said to be potential if and only if $\partial_{m_2} f_1 = \partial_{m_1} f_2$. For the linear model, this condition is equivalent to $a_{12} = a_{21}$. If $a_{12} \neq a_{21}$, the coupling is non-potential. we decompose the matrix $DA$ as $ DA = DA_s + DA_{\rm sk}$,  where $DA_{\rm sk} = \frac{DA - A^TD}{2}$, and $DA_s$ denotes the symmetric part of $DA$ and $DA_{\rm sk}$ its skew-symmetric part.

For every $\boldsymbol\mu \in \mathbb{R}^2$ we have $ \boldsymbol\mu^{\mathsf T} DA_{\mathrm{sk}} \boldsymbol\mu = 0$. Consequently, we get 
\[
\boldsymbol\mu^{\mathsf T} DA \boldsymbol\mu = \boldsymbol\mu^{\mathsf T} DA_s \boldsymbol\mu.
\]
Moreover, since $(DA_{\mathrm{sk}})^{\mathsf T} = -DA_{\mathrm{sk}}$, it follows that
\[
\boldsymbol\mu^{\mathsf T} DA_{\mathrm{sk}} \boldsymbol\mu
=
\left(
\boldsymbol\mu^{\mathsf T} DA_{\mathrm{sk}} \boldsymbol\mu
\right)^{\mathsf T}
=
\boldsymbol\mu^{\mathsf T} (DA_{\mathrm{sk}})^{\mathsf T} \boldsymbol\mu
=
- \boldsymbol\mu^{\mathsf T} DA_{\mathrm{sk}} \boldsymbol\mu.
\]
Thus, the scalar quantity $\boldsymbol\mu^{\mathsf T} DA_{\mathrm{sk}} \boldsymbol\mu$ equals its own additive inverse and hence must be zero.

The proof of Theorem~\ref{thm:twopopturnpike} remains valid in the case of non-potential couplings. Indeed, the energy method relies  on the coercivity of the symmetric part $DA_s$ of the coupling matrix. Consequently, if condition \eqref{matrixS2} holds with $DA_s$, then the local exponential turnpike property for the two-population system continues to hold, independently of whether the compatibility condition
$\partial_{m_2} f_1 = \partial_{m_1} f_2$ is satisfied. Moreover, in the energy identity, the linearized coupling appears through the quadratic form $\displaystyle
\int_{\mathbb T^N} \boldsymbol\mu^{T} DA \boldsymbol\mu \, dx$, which can be rewritten as
$\displaystyle
\int_{\mathbb T^N} \boldsymbol\mu^{T} DA_s \boldsymbol\mu \, dx$.
Consequently, the dissipative estimate from Lemma~\ref{lemdissipationid} is unchanged by the skew-symmetric component of $A$ and thus continues to hold in this context. The nonlinear remainder terms involve only the $C^2$-bounds of $\mathbf f$ and do not rely on any potential (gradient) structure. Therefore, the complete proof of Theorem~\ref{thm:twopopturnpike} extends directly, without modification, to this more general non-potential setting.

\begin{proposition}\label{prop:oscillations}
Consider the homogeneous linearized system with
$\bar m_1=\bar m_2=1$, $\bar u_1=\bar u_2=0$, and
$A=\alpha I_2+\beta J$, where
$$
J=
\begin{pmatrix}
0&1\\
-1&0
\end{pmatrix}.
$$
In this setting we have $A_s=\alpha I_2$. The matrix coercivity condition holds if and only if $\alpha>-\lambda_{\mathbb T}$. Moreover, when $\beta\neq0$, the linearized dynamics admits complex (non-real) characteristic exponents, and therefore the turnpike convergence toward the stationary state may exhibit oscillatory behavior.
\end{proposition}

\begin{proof}
Since $J^T=-J$, it follows that $A_s=\alpha I_2$. By Proposition~\ref{prop:sharp-threshold}, coercivity holds precisely when $\alpha+\lambda_{\mathbb T}>0$. For a Fourier mode $e^{ik\cdot x}$ with $k\neq0$, set $q=|k|^2$. The linearized system then reads
$$
\dot{\widehat v}_k=q\widehat v_k-A\widehat\mu_k,
\qquad
\dot{\widehat\mu}_k=-q\widehat v_k-q\widehat\mu_k.
$$
The matrix $J$ has eigenvalues $si$ with $s=\pm1$. On the associated complex eigenspace, $A$ acts as multiplication by $ \alpha+s i\beta$. Hence, the characteristic polynomial is
$
\rho^2=q(q+\alpha+s i\beta),
$
so that
$
\rho_{k,s}^{\pm}
=
\pm
\sqrt{q(q+\alpha+s i\beta)}$.
If $\beta\neq0$, the associated eigenvalues are complex and non-real. Consequently, the corresponding stable and unstable modes may exhibit rotational (oscillatory) dynamics while their magnitudes either decay or grow. Therefore, the skew-symmetric component alters the qualitative character of the transient behavior, but does not modify the coercivity threshold.
\end{proof}

\begin{lemma}\label{lem:ordspectrum}
The positivity of the standard eigenvalues of the non-symmetric coupling matrix $A$ does not, in general, ensure that the equal-weight matrix coercivity condition \eqref{matrixS2} is satisfied.  However, one can choose suitable weights $d_1,d_2>0$
 so that the weighted matrix coercivity condition is satisfied.
\end{lemma}

\begin{proof}
Consider $\displaystyle
A=
\begin{pmatrix}
1 & M\\
0 & 1
\end{pmatrix}$. The matrix $A$ has $1$ as its only eigenvalue, regardless of the value of $M$. On the other hand, choosing $D=I_2$  yields  that the symmetric part of $DA$ is
$ DA_s=
\begin{pmatrix}
1 & M/2\\
M/2 & 1
\end{pmatrix}$, whose smallest eigenvalue equals $\lambda_{\min}(DA_s)=1-\frac{M}{2}$. Furthermore, if the inequality $M > 2\bigl(1 + \lambda_{\mathbb T}\bigr)$ holds, then we obtain $\lambda_{\min}(DA_s) < -\lambda_{\mathbb T}$. By Proposition~\ref{prop:sharp-threshold}, this implies that condition~\eqref{matrixS2} fails in the homogeneous case. Thus, the usual spectrum of $A$ alone does not suffice to ensure that the equal-weight turnpike coercivity condition is fulfilled.  Whereas, if we choose  $D=\text{diag}\{1,d\}$, $d>\frac{M^2}{4},$
one has
\[
\frac{D(A+\lambda_{\mathbb T}I_2)
+(A+\lambda_{\mathbb T}I_2)^{T}D}{2}
=
\begin{pmatrix}
1&M/2\\
M/2&d
\end{pmatrix}>0.
\]
It follows that condition \eqref{matrixS2} can be satisfied by choosing suitable weights thanks to Proposition \ref{prop:sharp-threshold}.

\end{proof}

\begin{remark}
Assume $\bar m_1=\bar m_2=1$, $\bar u_1=\bar u_2=0$, and
$f_1(m_1,m_2)=a_{11}m_1+a_{12}m_2$, $f_2(m_1,m_2)=a_{21}m_1+a_{22}m_2$. Under these assumptions, the two-population matrix coercivity condition \eqref{matrixS2} is satisfied if and only if
$$
DA_s+\lambda_{\mathbb T}D>0,
$$
that is, precisely when $a_{11}+\lambda_{\mathbb T}>0$, $a_{22}+\lambda_{\mathbb T}>0$, and
$$
(a_{11}+\lambda_{\mathbb T})(a_{22}+\lambda_{\mathbb T})
>
\frac{(d_1a_{12}+d_2a_{21})^2}{4d_1d_2}.
$$
Whenever these inequalities are fulfilled, the local exponential turnpike theorem becomes applicable. In this framework, the antisymmetric component
$$
\frac12
\begin{pmatrix}
0&d_1a_{12}-d_2a_{21}
\\
d_2a_{21}-d_1a_{12}&0
\end{pmatrix}
$$
does not affect the coercivity threshold, although it may generate oscillatory transient behavior. Moreover, the equivalence between \eqref{matrixS2} and the condition $DA_s+\lambda_{\mathbb T}D>0$ is proven in Proposition~\ref{prop:sharp-threshold}. The latter is exactly the usual positive-definiteness condition for a $2\times2$ symmetric matrix. Once these inequalities are satisfied, condition \eqref{matrixS2} is fulfilled, and Theorem~\ref{thm:twopopturnpike} ensures the local exponential turnpike property. The final assertion then follows directly from Proposition~\ref{prop:oscillations}.  In addition, by choosing appropriate diagonal weights, the matrix coercivity condition (\ref{matrixS2}) can be recovered even when the equal-weight condition fails. This shows that the introduction of the weight matrix $D$ enlarges the range of applicability of the coercivity criterion.
\end{remark}

\section{Conclusion and Discussion}
In this paper, we focus on the local turnpike stability criterion for two-population mean-field game models with the quadratic Hamiltonian on the torus  $\mathbb T^N$. Under suitable assumptions, we construct a weighted Lyapunov functional and establish the local turnpike stability property through a bootstrap and fixed-point argument, assuming smallness and regularity conditions on the initial and terminal data.  More precisely, we propose the matrix coercive condition \eqref{matrixS2}, which guarantees the local turnpike stability of the ergodic solutions under suitable structural assumptions.  Moreover, we establish a spectral characterization of condition \eqref{matrixS2} and show that it coincides with the eigenvalue problem associated with the linearization of a corresponding semilinear system.  Under suitable regularity and positivity assumptions on the admissible set, we investigate the relationship between the coercivity condition and the second variation of the energy functional associated with potential two-population mean-field game models. With this, we find a nondegenerate local minimizer to the potential mean-field games model is locally turnpike stable. We also discuss several interesting and challenging open problems.  It would be worthwhile to extend the stability criterion to general Hamiltonians. However, such an extension appears challenging due to the need for new Lyapunov functionals and the lack of suitable regularity results of Hamilton-Jacobi equations.

\bibliographystyle{plainurl}
\bibliography{Refsmultipopuplation}
\end{document}